\documentclass[11pt]{article}
\usepackage{amsfonts}
\usepackage{latexsym}
\usepackage{amsmath}
\usepackage{amssymb}
\usepackage{amsthm}
\usepackage{amsmath}

\newtheorem{thm}{Theorem}[section]
\newtheorem*{thm*}{Theorem}
\newtheorem{cor}[thm]{Corollary}
\newtheorem{corollary}[thm]{Corollary}
\newtheorem{example}[thm]{Example}
\newtheorem{lem}[thm]{Lemma}

\newtheorem{prop}[thm]{Proposition}
\newtheorem*{prop*}{Proposition}

\newtheorem*{conj*}{Conjecture}

\newtheorem*{dfn*}{Definition}
\theoremstyle{definition}
\newtheorem{rem}[thm]{\textbf{Remark}}
\newtheorem*{rmk*}{Remark}
\newtheorem*{fact*}{Fact}

\theoremstyle{proof}

\numberwithin{equation}{section}

\newcommand{\vol}{\textrm{Vol}}
\newcommand{\norm}[1]{\left\Vert#1\right\Vert}

\newcommand{\abs}[1]{\left\vert#1\right\vert}
\newcommand{\set}[1]{\left\{#1\right\}}
\newcommand{\brac}[1]{\left(#1\right)}
\newcommand{\scalar}[1]{\left \langle #1 \right \rangle}

\newcommand{\Real}{\mathbb{R}}

\newcommand{\eps}{\epsilon}
\newcommand{\I}{\mathcal{I}}
\newcommand{\Ric}{\mbox{\rm{Ric}}}
\newcommand{\II}{\mbox{\rm{II}}}
\newcommand{\Jac}{\text{Jac}}
\renewcommand{\div}{\text{div}}
\renewcommand{\S}{\mathcal{S}}
\newcommand{\Var}{\text{\rm{Var}}}
\newcommand{\CD}{\text{CD}}

\def\Xint#1{\mathchoice
{\XXint\displaystyle\textstyle{#1}}{\XXint\textstyle\scriptstyle{#1}}{\XXint\scriptstyle\scriptscriptstyle{#1}}{\XXint\scriptscriptstyle\scriptscriptstyle{#1}}\!\int}
\def\XXint#1#2#3{{\setbox0=\hbox{$#1{#2#3}{\int}$}
\vcenter{\hbox{$#2#3$}}\kern-.5\wd0}}

\def\dashint{\Xint-}

\begin{document}

\date{}

\title{Poincar\'e and Brunn--Minkowski inequalities on the boundary of weighted Riemannian manifolds}
\author{Alexander V. Kolesnikov\textsuperscript{1} and Emanuel Milman\textsuperscript{2}}

\footnotetext[1]{Faculty of Mathematics, Higher School of Economics, Moscow, Russia. 
The author was supported by RFBR project 17-01-00662 and DFG project RO 1195/12-1. The article was prepared within the framework of the Academic Fund Program at the National Research University Higher School of Economics (HSE) in 2017-2018  (grant No 17-01-0102) and by the Russian Academic Excellence Project ``5-100". Emails: akolesnikov@hse.ru, sascha77@mail.ru.}

\footnotetext[2]{Department of Mathematics, Technion - Israel
Institute of Technology, Haifa 32000, Israel. Supported by BSF (grant no. 2010288).
The research leading to these results is part of a project that has received funding from the European Research Council (ERC) under the European Union's Horizon 2020 research and innovation programme (grant agreement No 637851). Email: emilman@tx.technion.ac.il.}

\maketitle

\begin{abstract}
We study a Riemannian manifold equipped with a density which satisfies the Bakry--\'Emery Curvature-Dimension condition (combining a lower bound on its generalized Ricci curvature and an upper bound on its generalized dimension). We first obtain a Poincar\'e-type inequality on its boundary assuming that the latter is locally-convex; this generalizes a purely Euclidean inequality of Colesanti, originally derived as an infinitesimal form of the Brunn-Minkowski inequality, thereby precluding any extensions beyond the Euclidean setting.  
A dual version for generalized mean-convex boundaries is also obtained, yielding spectral-gap estimates for the weighted Laplacian on the boundary. Motivated by these inequalities, a new geometric evolution equation is proposed, which extends to the Riemannian setting the Minkowski addition operation of convex domains, a notion thus far confined to the purely linear setting. This geometric flow is characterized by having parallel normals (of varying velocity) to the evolving hypersurface along the trajectory, and is intimately related to a homogeneous Monge-Amp\`ere equation on the exterior of the convex domain. Using the aforementioned Poincar\'e-type inequality on the boundary of the evolving hypersurface, we obtain a novel Brunn--Minkowski inequality in the weighted-Riemannian setting, amounting to a certain concavity property for the weighted-volume of the evolving enclosed domain. All of these results appear to be new even in the classical non-weighted Riemannian setting. 
\end{abstract}

\section{Introduction}

Throughout the paper we consider a compact \emph{weighted-manifold} $(M,g,\mu)$, namely a compact smooth complete connected and oriented $n$-dimensional Riemannian manifold $(M,g)$ with boundary $\partial M$, equipped with a measure:
\[
\mu = \exp(-V) d {\vol}_M ~,
\] 
where $\vol_M$ is the Riemannian volume form on $M$ and $V \in C^2(M)$ is twice continuously differentiable. The boundary $\partial M$ is assumed to be a $C^2$ manifold with outer unit-normal $\nu = \nu_{\partial M}$. The corresponding symmetric diffusion operator with invariant measure $\mu$, which is called the weighted-Laplacian, is given by:
\[
L  = L_{(M,g,\mu)} := \exp(V) \div_g ( \exp(-V) \nabla_g) = \Delta_g - \scalar{\nabla_g V,\nabla_g} ~,
\]
where $\scalar{\cdot,\cdot}$ denotes the Riemannian metric $g$,  $\nabla = \nabla_g$ denotes the Levi-Civita connection, $\div = \div_g =  tr(\nabla \cdot)$ denotes the Riemannian divergence operator, and $\Delta = \Delta_g = \div_g \nabla_g$ is the Laplace-Beltrami operator.  Indeed, note that with these generalized notions, the usual integration by parts formula is satisfied for $f,h \in C^2(M)$: \[
\int_M (Lf) h d\mu = \int_{\partial M} f_\nu h d\mu_{\partial M} - \int_M \scalar{\nabla f,\nabla h} d\mu = \int_{\partial M} (f_\nu h - h_\nu f) d\mu_{\partial M} +  \int_M (Lh) f d\mu ~,
\]
where  $u_\nu = \nu \cdot u$ and $\mu_{\partial M} := \exp(-V) \vol_{\partial M}$.

The second fundamental form $\II = \II_{\partial M}$ of $\partial M \subset M$ at $x \in \partial M$ is as usual (up to sign) defined by $\II_x(X,Y) = \scalar{\nabla_X \nu, Y}$, $X,Y \in T_x \partial M$. The quantities
\[
H_g(x) := tr(\II_x) ~,~ H_\mu (x) := H_g(x) - \scalar{\nabla V(x) ,\nu(x)} ~,
\]
are called the Riemannian mean-curvature and \emph{generalized} mean-curvature of $\partial M$ at $x \in \partial M$, respectively. It is well-known that $H_g$ governs the first variation of $\vol_{\partial M}$ under the normal-map $t \mapsto \exp(t \nu)$, and similarly $H_\mu$ governs the first variation of $\exp(-V)\vol_{\partial M} = \mu_{\partial M}$ in the weighted-manifold setting, see e.g. \cite{EMilmanGeometricApproachPartI} or Subsection \ref{subsec:Full-BM}. 

\medskip

In the purely Riemannian setting, it is classical that positive lower bounds on the Ricci curvature tensor $\Ric_g$ and upper bounds on the topological dimension $n$ play a fundamental role in governing various Sobolev-type inequalities on $(M,g)$, see e.g. \cite{ChavelEigenvalues,GallotBourbaki,GallotIsoperimetricInqs,LiYauEigenvalues,YauIsoperimetricConstantsAndSpectralGap} and the references therein. In the weighted-manifold setting, the pertinent information on \emph{generalized} curvature and \emph{generalized} dimension may be incorporated into a single tensor, which was put forth by Bakry and \'Emery \cite{BakryEmery,BakryStFlour} (cf. Lott \cite{LottRicciTensorProperties}) following Lichnerowicz \cite{Lichnerowicz1970GenRicciTensorCRAS,Lichnerowicz1970GenRicciTensor}. Setting $\Psi = \exp(-V)$, the $N$-dimensional generalized Ricci curvature tensor
($N \in (-\infty,\infty]$) is defined as:
\begin{equation} \label{eq:Ric-def}
\mbox{\rm{Ric}}_{\mu,N} := \rm{Ric}_g + \nabla^2 V - \frac{1}{N-n} d V\otimes d V = \rm{Ric}_g - (N-n) \frac{\nabla^2  \Psi^{\frac{1}{N-n}}}{\Psi^{\frac{1}{N-n}}} ~, 
\end{equation}
where  $\nabla^2 = \nabla_g^2$ denotes the Riemannian Hessian operator. Note that the case $N=n$ is only defined when $V$ is constant, i.e. in the classical non-weighted Riemannian setting where $\mu$ is proportional to $\vol_M$, in which case $\Ric_{\mu,n}$ boils down to classical Ricci tensor $\Ric_g$. When $N= \infty$ we set:
\[
\Ric_\mu := \Ric_{\mu,\infty} = \Ric_g + \nabla^2 V ~.
\]
The Bakry--\'Emery Curvature-Dimension condition $\CD(\rho,N)$, $\rho \in \Real$, is then the requirement that as 2-tensors on $M$:
\[
\mbox{\rm{Ric}}_{\mu,N} \geq \rho g ~.
\]
Consequently, $\rho$ is interpreted as a generalized-curvature lower bound, and $N$ as a generalized-dimension upper bound.
The $\CD(\rho,N)$ condition has been an object of extensive study over the last two decades (see e.g. also \cite{QianWeightedVolumeThms,LedouxLectureNotesOnDiffusion,CMSInventiones, CMSManifoldWithDensity, VonRenesseSturmRicciChar,BakryQianGenRicComparisonThms,WeiWylie-GenRicciTensor,MorganBook4Ed,EMilmanSharpIsopInqsForCDD, BGL-Book} and the references therein), especially since Perelman's work on the Poincar\'e Conjecture \cite{PerelmanEntropyFormulaForRicciFlow}, and the extension of the Curvature-Dimension condition to the metric-measure space setting by Lott--Sturm--Villani \cite{SturmCD12,LottVillaniGeneralizedRicci}.

\medskip
Recalling the interpretation of $N$ as an upper bound on the generalized-dimension, it is customary in the literature to only treat the case when $N \in [n,\infty]$;
however, our methods will also apply with no extra effort to the case when $N \in (-\infty,0]$, and so our results are treated in this greater generality, which in the Euclidean setting encompasses the entire class of Borell's convex (or ``$1/N$-concave") measures \cite{BorellConvexMeasures} (cf. \cite{BrascampLiebPLandLambda1,BobkovLedouxWeightedPoincareForHeavyTails}). It will be apparent that the more natural parameter is actually $1/N$, with $N=\infty,0$ interpreted as $1/N = 0,-\infty$, respectively, and so our results hold in the range $1/N \in [-\infty,1/n]$. As $d V\otimes d V$ appearing in (\ref{eq:Ric-def}) is a positive semi-definite tensor, the $\CD(\rho,N)$ condition is clearly monotone in $\frac{1}{N-n}$ and hence in $\frac{1}{N}$ in the latter range, so for all  $N_+ \in [n,\infty], N_- \in (-\infty,0]$:
\[
\CD(\rho,n) \Rightarrow \CD(\rho,N_+) \Rightarrow \CD(\rho,\infty) \Rightarrow \CD(\rho,N_-) \Rightarrow \CD(\rho,0) ~;
\]
note that $\CD(\rho,0)$ is the weakest condition in this hierarchy. 
 It seems that outside the Euclidean setting, this extension of the Curvature-Dimension condition to the range $N \in (-\infty,0]$ has not attracted much attention in the weighted-Riemannian and more general metric-measure space setting (cf. \cite{SturmCD12,LottVillaniGeneralizedRicci}); an exception is the work of Ohta and Takatsu \cite{OhtaTakatsuEntropies1,OhtaTakatsuEntropies2}. We expect this gap in the literature to be quickly filled  (in fact, concurrently to posting this work along with its companion paper \cite{KolesnikovEMilmanReillyPart1} on the arXiv, Ohta \cite{Ohta-NegativeN} has posted a first attempt of a systematic treatise of the range $N \in (-\infty,0]$, and subsequently other authors have also begun treating this extended range \cite{EMilmanNegativeDimension,KlartagLocalizationOnManifolds,Wylie-SectionalCurvature,KennardWylie-WeightedSectionalCurvature,EMilmanHarmonicMeasures}). 

\subsection{Poincar\'e-type inequalities on $\partial M$}

Our first main result in this work is the following Poincar\'e-type inequality on $\partial M$, which appears to be novel in the Riemannian setting (even in the classical non-weighted setting, i.e. $N=n$):

\begin{thm}[Generalized Colesanti Inequality] \label{thm:Colesanti-intro}
Assume that $(M,g,\mu)$ satisfies the $\CD(0,N)$ condition ($1/N \in [-\infty,1/n]$) and that $\II_{\partial M} > 0$ ($M$ is strictly locally-convex).
Then the following inequality holds for any $f \in C^1(\partial M)$:
\begin{equation} \label{eq:Colesanti-intro}
 \int_{\partial M} H_\mu f^2 d\mu_{\partial M} - \frac{N-1}{N}\frac{\brac{\int _{\partial M} f d\mu_{\partial M}}^2}{\mu(M)} \leq  \int_{\partial M} \scalar{\II_{\partial M}^{-1} \;\nabla_{\partial M} f,\nabla_{\partial M} f} d\mu_{\partial M} ~,
\end{equation}
where $\nabla_{\partial M}$ denotes the induced Levi-Civita connection on the boundary. 
\end{thm}

Theorem \ref{thm:Colesanti-intro} was obtained by A.~Colesanti in \cite{ColesantiPoincareInequality} with $N=n$ for a compact subset $M$ of \emph{Euclidean space} $\Real^n$ having a $C^2$ strictly convex boundary and endowed with the Lebesgue measure ($V=0$). Colesanti derived this inequality as an infinitesimal version of the (purely Euclidean) Brunn-Minkowski inequality, and so his method is naturally \emph{confined to the Euclidean setting}. In contrast, we derive Theorem \ref{thm:Colesanti-intro} in Section \ref{sec:Col} directly by an $L^2(\mu)$-duality argument coupled with a generalized version of the classical Reilly formula (derived in our companion work \cite{KolesnikovEMilmanReillyPart1} following Li and Du \cite{MaDuGeneralizedReilly}). In particular, this yields a new Riemannian proof of the Brunn--Minkowski inequality for convex domains in Euclidean space (as well as its generalization by Borell \cite{BorellConvexMeasures} and Brascamp--Lieb \cite{BrascampLiebPLandLambda1}, see Subsections \ref{subsec:intro-BM} and \ref{subsec:BBL}). 

\medskip

We also obtain a dual-version of Theorem \ref{thm:Colesanti-intro}, which in fact applies to mean-convex domains: 
\begin{thm}[Dual Generalized Colesanti Inequality] \label{thm:dual-Colesanti-intro}
Assume that $(M,g,\mu)$ satisfies the $\CD(\rho,0)$ condition, $\rho \in \Real$, and that $H_\mu > 0$ on $\partial M$ ($M$ is strictly generalized mean-convex). Then for any $f \in C^{2}(\partial M)$ and $C \in \Real$:
\[
\int_{\partial M} \scalar{\II_{\partial M}  \;\nabla_{\partial M} f,\nabla_{\partial M} f} d\mu_{\partial M}  \leq \int_{\partial M} \frac{1}{H_\mu} \Bigl(L_{\partial M} f + \frac{\rho (f-C) }{2} \Bigr)^2 d\mu_{\partial M}~,
\]
where $L_{\partial M} = L_{(\partial M, g|_{\partial M}, \mu_{\partial M})}$ denotes the weighted-Laplacian on the boundary. 
\end{thm}

By specializing to the constant function $f \equiv 1$, various mean-curvature inequalities for convex and mean-convex boundaries of $\CD(0,N)$ weighted-manifolds are obtained in Section \ref{sec:Col-App}, immediately recovering (when $N \in [n,\infty]$) and extending (when $N \leq 0$) recent results of Huang--Ruan \cite{HuangRuanMeanCurvatureEstimates}. Under various combinations of non-negative lower bounds on $H_\mu$, $\II_{\partial M}$ and $\rho$, spectral-gap estimates on convex boundaries of $\CD(\rho,0)$ weighted-manifolds are deduced in Sections \ref{sec:Col-App} and \ref{sec:boundaries}. For instance, we show:

\begin{thm} \label{thm:IIHRho-Poincare-intro}
Assume that $(M,g,\mu)$ satisfies $\CD(\rho,0)$, $\rho \geq 0$, 
and that $\II_{\partial M} \ge \sigma g|_{\partial M}$, $H_\mu \ge \xi$ on $\partial M$ with $\sigma,\xi > 0$. Then for all $f \in C^1(\partial M)$ with $\int_{\partial M} f d\mu_{\partial M} = 0$: 
\[\lambda_1 \int_{\partial M} f^2 d\mu_{\partial M} \leq \int_{\partial M} |\nabla_{\partial M} f|^2 d \mu_{\partial M} ~,
\]where:
\[
\lambda_1 \geq \frac{\rho + a + \sqrt{2 a \rho + a^2}}{2} \geq \max\brac{a,\frac{\rho}{2}}  ~,~ a := \sigma \xi ~.
\]
\end{thm}
\noindent
This extends and refines the estimate $\lambda_1 \geq (n-1) \sigma^2$ of Xia \cite{XiaSpectralGapOnConvexBoundary} in the classical non-weighted Riemannian setting ($V \equiv 0$) when $\Ric_g \geq 0$ ($\rho=0$), since in that case $\xi \geq (n-1) \sigma$. 
Other spectral-gap and log-Sobolev estimates are obtained in Section \ref{sec:boundaries}, by taking note that the boundary $(\partial M,g|_{\partial M}, \mu_{\partial M})$ satisfies the $\CD(\rho_0,N-1)$ condition for an appropriate $\rho_0$. For instance, Theorem \ref{thm:Colesanti-intro} yields the following estimate (see Theorem \ref{thm:Col-On-Boundary} and Remark \ref{rem:Col-On-Boundary}):

\begin{thm} \label{thm:intro-Col-On-Boundary}
Let $(M^n,g,\mu)$ satisfy $\CD(0,0)$, $(\partial M,g|_{\partial M},\mu_{\partial M})$ satisfy $\CD(0,\infty)$, and assume that:
\[
\II_{\partial M} \geq \sigma g|_{\partial M} ~,~ H_\mu \geq \xi ~,
\]
for some positive measurable functions $\sigma,\xi : \partial M \rightarrow \Real_+$. Then for all $f \in C^1(\partial M)$ with $\int_{\partial M} f d\mu_{\partial M} = 0$:  
\[
\int_{\partial M} f^2 d\mu_{\partial M}\leq C \brac{\dashint_{\partial M} \frac{1}{\xi} d\mu_{\partial M} \dashint_{\partial M} \frac{1}{\sigma} d\mu_{\partial M}} \int_{\partial M} \abs{\nabla_{\partial M} f}^2 d\mu_{\partial M} ~,
\]
where $C>1$ is some universal (dimension-independent) numeric constant.
\end{thm}

\subsection{Connections to the Brunn--Minkowski Theory} \label{subsec:intro-BM}

Recall that the classical Brunn--Minkowski inequality in Euclidean space \cite{Schneider-Book,GardnerSurveyInBAMS} asserts that:
\begin{equation} \label{eq:BM-intro}
\vol((1-t) K + t L)^{1/n} \geq (1-t) \vol(K)^{1/n} + t \vol(L)^{1/n} ~,~ \forall t \in [0,1] ~,
\end{equation}
for all convex $K,L \subset \Real^n$; it was extended to arbitrary Borel sets by Lyusternik. Here $\vol$ denotes Lebesgue measure and 
$A + B := \set{a + b \; ; \; a \in A , b \in B}$ denotes Minkowski addition. We refer to the excellent survey by R. Gardner \cite{GardnerSurveyInBAMS} for additional details and references. 

For convex sets, (\ref{eq:BM-intro}) is equivalent to the concavity of the function $t \mapsto \vol(K + t L)^{1/n}$. By Minkowski's theorem, extending Steiner's observation for the case that $L$ is the Euclidean ball, $\vol(K+t L)$ is an $n$-degree polynomial $\sum_{i=0}^n {n \choose i} W_{n-i}(K,L) t^i$, whose coefficients
\begin{equation} \label{eq:W-def}
 W_{n-i}(K,L) := \frac{(n-i)!}{n!} \brac{\frac{d}{dt}}^{i} \vol(K + t L)|_{t=0} ~,
\end{equation}
are called mixed-volumes. The above concavity thus amounts to the following ``Minkowski's second inequality", which is a particular case of the Alexandrov--Fenchel inequalities:
\begin{equation} \label{eq:Mink-II}
W_{n-1}(K,L)^2 \geq W_{n-2}(K,L) W_n(K,L) = W_{n-2}(K,L)  \vol(K) ~.
\end{equation}
It was shown by Colesanti \cite{ColesantiPoincareInequality} that (\ref{eq:Mink-II}) is equivalent to (\ref{eq:Colesanti-intro}) in the Euclidean setting. In fact, a Poincar\'e-type inequality on the sphere, which is a reformulation of (\ref{eq:Colesanti-intro}) obtained via the Gauss-map, was established already by Hilbert (see \cite{BusemannConvexSurfacesBook,HormanderNotionsOfConvexityBook}) in his proof of (\ref{eq:Mink-II}) and thus the Brunn--Minkowski inequality for convex sets. Going in the other direction, the Brunn--Minkowski inequality was used by Colesanti to establish (\ref{eq:Colesanti-intro}). 
See e.g. \cite{BrascampLiebPLandLambda1,BobkovLedoux, Ledoux-Book} for further related connections. 

In view of our generalization of (\ref{eq:Colesanti-intro}) to the weighted-Riemannian setting, it is all but natural to wonder whether there is a Riemannian Brunn--Minkowski theory lurking in the background. Note that when $L$ is the Euclidean unit-ball $D$, then $K+ t D$ coincides with $K_t := \set{ x \in \Real^n \; ; \; d(x,K) \leq t }$, where $d$ is the Euclidean distance. The corresponding distinguished mixed-volumes $W_{n-i}(K) = W_{n-i}(K,D)$, which are called intrinsic-volumes or quermassintegrals, are obtained (up to normalization factors) as the $i$-th variation of $t \mapsto \vol(K_t)$.
Analogously, we may define $K_t$ on a general Riemannian manifold with $d$ denoting the geodesic distance, and given $1/N \in (-\infty,1/n]$, define the following \emph{generalized} quermassintegrals of $K$ as the $i$-th variations of $t \mapsto \mu(K_t)$, $i=0,1,2$ (up to normalization):
\[
W_N(K) := \mu(K) ~,~ W_{N-1}(K) := \frac{1}{N} \int_{\partial K} d\mu_{\partial K}  ~,~ W_{N-2}(K) := \frac{1}{N(N-1)} \int_{\partial K} H_\mu d\mu_{\partial K} ~.
\]
Applying (\ref{eq:Colesanti-intro}) to the constant function $f \equiv 1$, we obtain in Section \ref{sec:BM} the following interpretation of the resulting inequality:

\begin{cor}{(Riemannian Geodesic Brunn-Minkowski for Convex Sets)}
Let $K$ denote a compact subset of $(M^n,g)$ having $C^2$ smooth and strictly-convex boundary ($\II_{\partial K} > 0$), which is bounded away from $\partial M$. Assume that $(K,g|_K,\mu|_K)$ satisfies the $\CD(0,N)$ condition ($1/N \in (-\infty,1/n]$). Then the following generalized Minkowski's second inequality for geodesic extensions holds:
\[
W_{N-1}(K)^2 \geq W_{N}(K) W_{N-2}(K) ~.
\]
Equivalently, $(d/dt)^2 N \mu(K_t)^{1/N} |_{t=0} \leq 0$ where $K_t := \set{ x \in M \; ; \; d(x,K) \leq t }$.

\noindent
In other words, if $(M,g,\mu)$ satisfies $\CD(0,N)$, then the function:
\[
t \mapsto N \mu(K_t)^{1/N}
\]
is concave on any interval $[0,T]$ so that for all $t \in [0,T)$, $K_t$ remains $C^2$ smooth, strictly locally-convex,
and bounded away from $\partial M$.
\end{cor}

A much greater challenge is to find an extension of the Minkowski sum $K + t L$ for a general convex $L$ \emph{beyond the linear setting}. Observe that due to lack of homogeneity, this is \emph{not} the same as extending the operation of Minkowski interpolation $(1-t)K + t L$, a trivial task on any geodesic metric space by using geodesic interpolation. Motivated by the equivalence between (\ref{eq:Colesanti-intro}) and (\ref{eq:Mink-II}) which should persist in the weighted-Riemannian setting, we propose in Section \ref{sec:BM} a Riemannian generalization of $K + t L$ based on a novel geometric flow. Given a strictly locally-convex $K \subset (M,g)$ and $\varphi  \in C^2(\partial K)$, this flow produces a set denoted by $K_t = K_{\varphi,t} := K + t \varphi$, by evolving its boundary $\partial K_t = F_t(\partial K)$ along the map 
$F_t : \partial K \mapsto M$ solving:
\begin{eqnarray*}
\frac{d}{dt} F_t(y)  &=& \omega_t(F_t(y)) ~,~ F_0 = Id~,~ y \in \partial K ~,~ t \in [0,T] ~, \\
 \omega_t &:=& \varphi_t \nu_{\partial K_t} + \tau_t ~ \text{ on $\partial K_t$} ~,~ \tau_t := \II_{\partial K_t}^{-1} \nabla_{\partial K_t} \varphi_t  ~,~ \varphi_t := \varphi \circ F_t^{-1} ~.
\end{eqnarray*}
Indeed, we show that this operation coincides in the Euclidean setting with the usual Minkowski sum $K + t L$, when $\varphi$ is chosen to be the support function of $L$ composed with the Gauss map on $\partial K$. We dub this flow the 
``Parallel Normal Flow", since it is precisely characterized by having parallel normals (of spatially varying normal velocity $\varphi$) to the evolving hypersurface along the trajectory. Some additional interesting connections to an appropriate homogeneous Monge-Amp\`ere equation are discussed in Section \ref{sec:BM}. We do not go here into justifications for the existence of such a flow on an interval $[0,T]$
(except when all of the data is analytic, in which case the short-time existence is easy to justify), 
but rather observe the following new:
\begin{thm}[Riemannian Brunn--Minkowski Inequality] \label{thm:intro-Full-BM}
Let $K$ denote a compact subset of $(M^n,g)$ having $C^2$ smooth and strictly-convex boundary ($\II_{\partial K} > 0$), which is bounded away from $\partial M$. 
Let $\varphi \in C^2(\partial K)$, and let $K_t$ denote the Riemannian Minkowski extension $K_t := K + t \varphi$, assuming it is well-posed for all $t \in [0,T]$. Assume that $(M,g,\mu)$ satisfies the $\CD(0,N)$ condition ($1/N \in (-\infty,1/n]$). Then the function:
\[
t \mapsto N \mu(K_t)^{1/N} 
\]
is concave on $[0,T]$. 
\end{thm}

It turns out that the latter concavity is equivalent to our generalized Colesanti inequality (\ref{eq:Colesanti-intro}). In view of the remarks above, Theorem \ref{thm:intro-Full-BM} is interpreted as a version of the Brunn--Minkowski inequality in the weighted Riemannian setting. Furthermore, this leads to a natural way of defining the mixed-volumes of $K$ and $\varphi$ in this setting, namely as variations of $t \mapsto \mu(K + t \varphi)$. Yet another natural flow producing the aforementioned concavity is also suggested in Section \ref{sec:BM}; however, this flow does not seem to produce Minkowski summation in the Euclidean setting. See Remark \ref{rem:other-gen-BM} for a comparison with other known extensions of the Brunn--Minkowski inequality to the metric-measure space setting, which are very different from the one above, as they are all based on the obvious extension of $(1-t) K + t L$ via geodesic interpolation.

\medskip

To conclude this work, we provide in Section \ref{sec:Apps} some further applications of our results to the study of isoperimetric inequalities on weighted Riemannian manifolds. Additional applications will be developed in a subsequent work. 

\medskip \noindent
\textbf{Acknowledgements.} We thank Franck Barthe, Bo Berndtsson, Andrea Colesanti, Dario Cordero-Erausquin, Bo'az Klartag, Michel Ledoux, Frank Morgan, Van Hoang Nguyen, Shin-ichi Ohta, Yehuda Pinchover and Steve Zelditch for their comments and interest.

\section{Preliminaries} \label{sec:prelim}

\subsection{Notation}

We denote by $int(M)$ the interior of $M$. Given a compact differentiable manifold $\Sigma$ (which is at least $C^k$ smooth), we denote by $C^{k}(\Sigma)$ the space of real-valued functions on $\Sigma$ with continuous derivatives $\brac{\frac{\partial}{\partial x}}^a f$, for every multi-index $a$ of order $|a| \leq k$ in a given coordinate system. Similarly, the space $C^{k,\alpha}(\Sigma)$ denotes the subspace of functions whose $k$-th order derivatives are H\"{o}lder continuous of order $\alpha$ on the $C^{k,\alpha}$ smooth manifold $\Sigma$. When $\Sigma$ is non-compact, we may use $C_{loc}^{k,\alpha}(\Sigma)$ to denote the class of functions $u$ on $M$ so that $u|_{\Sigma_0} \in C^{k,\alpha}(\Sigma_0)$ for all compact subsets $\Sigma_0 \subset \Sigma$. These spaces are equipped with their usual corresponding topologies. 

Throughout this work, when integrating by parts, we employ a slightly more general version of the textbook Stokes Theorem $\int_M d\omega = \int_{\partial M} \omega$,  in which one only assumes that $\omega$ is a continuous differential $(n-1)$-form on $M$ which is differentiable on $int(M)$ (and so that $d\omega$ is integrable there); a justification may be found in \cite{MacdonaldGeneralizedStokes}. This permits us to work with the classes $C^k_{loc}(int(M))$ occurring throughout this work. 

Given a finite measure $\nu$ on a measurable space $\Omega$, and a $\nu$-integrable function $f$ on $\Omega$, we denote:
\[
\dashint_\Omega f d\nu := \frac{1}{\nu(\Omega)} \int_\Omega f d\nu ~,~ \Var_\nu(f) := \int_{\Omega} \brac{f - \dashint_\Omega f d\nu}^2 d\nu ~. 
\]

Throughout this work we employ Einstein summation convention.
By abuse of notation, we denote different covariant and contravariant versions of a tensor in the same manner. So for instance, $\Ric_\mu$ may denote the $2$-covariant tensor $(\Ric_\mu)_{\alpha,\beta}$, but also may denote its $1$-covariant $1$-contravariant version $(\Ric_\mu)^{\alpha}_{\beta}$, as in:
\[
 \scalar{\Ric_\mu \nabla f , \nabla f} = g_{i,j} (\Ric_{\mu})^i_k \nabla^k f \nabla^j f  = (\Ric_\mu)_{i,j} \nabla^i f \nabla^j f = \Ric_\mu(\nabla f,\nabla f) ~.
 \]
 Similarly, reciprocal tensors are interpreted according to the appropriate context. For instance, the $2$-contravariant tensor $(\text{II}^{-1})^{\alpha,\beta}$ is defined by:
 \[
 (\text{II}^{-1})^{i,j} \text{II}_{j,k} = \delta^i_k ~.
 \]
We freely raise and lower indices by contracting with the metric when there is no ambiguity regarding which underlying metric is being used; this is indeed the case throughout this work, with the exception of Subsection \ref{subsec:strange-flow}. Since we mostly deal with $2$-tensors, the only possible contraction is often denoted by using the trace notation $tr$. 

In addition to the already mentioned notation in the weighted-Riemannian setting, we will also make use of $\div_{g,\mu} = \div_{(M,g,\mu)}$ to denote the weighted-divergence operator on the weighted-manifold $(M,g,\mu)$, so that if $\mu = \exp(-V) d\vol_{M}$ then:
\[
\text{div}_{g,\mu}(X) := \exp(V) \div_{g} (\exp(-V) X) = \div_{g}(X) - g(\nabla_{g} V , X) ~,~ \forall X \in T M ~;
\]
this is the natural notion of divergence in the weighted-manifold setting, satisfying the usual integration by parts formula (say if $M$ is closed):
\[
\int_{M} f \cdot \text{div}_{g,\mu}(X) d\mu = - \int_{M} g(\nabla_{g} f,X) d\mu ~,~\forall X \in T M ~.
\]

Finally, when studying consequences of the $\CD(\rho,N)$ condition, the various expressions in which $N$ appears are interpreted in the limiting sense when $1/N= 0$. For instance, $N/(N-1)$ is interpreted as $1$, and $N f^{1/N}$ is interpreted as $\log f$ (since $\lim_{1/N \rightarrow 0} N (x^{1/N} -1) = \log(x)$; the constant $-1$ in the latter limit does not influence our application of this convention). 
 
 \subsection{Generalized Reilly Formula}

Denote by $\S_0(M)$ the class of functions $u$ on $M$ which are $C^2$ smooth in the interior of $M$ and $C^1$ smooth on the entire compact $M$. 
Denote by $\S_N(M)$ the subclass of functions which in addition satisfy that $u_\nu$ is $C^1$ smooth on $\partial M$. 

\begin{thm}[Generalized Reilly Formula] \label{thm:Reilly}
For any function $u \in \S_N(M)$: 
\begin{multline}
\label{Reilly}
\int_M (L u)^2 d\mu = \int_M \norm{\nabla^2 u}^2 d\mu + \int_M \scalar{ \Ric_\mu \; \nabla u, \nabla u} d\mu + \\
\int_{\partial M} H_\mu (u_\nu)^2 d\mu_{\partial M} + \int_{\partial M} \scalar{\II_{\partial M}  \;\nabla_{\partial M} u,\nabla_{\partial M} u} d\mu_{\partial M} - 2 \int_{\partial M} \scalar{\nabla_{\partial M} u_\nu, \nabla_{\partial M} u} d\mu_{\partial M} ~.
\end{multline}
Here $\norm{\nabla^2 u}$ denotes the Hilbert-Schmidt norm of $\nabla^2 u$.
\end{thm}

This natural generalization of the (integrated) Bochner--Lichnerowicz--Weitzenb\"{o}ck formula for manifolds with boundary was first obtained by R.C.~Reilly \cite{ReillyOriginalFormula} in the classical Riemannian setting ($\mu=\vol_M$). The version above, proved in our companion work \cite{KolesnikovEMilmanReillyPart1}, is a variant of a prior generalized version due to M.~Li and S.-H.~Du \cite{MaDuGeneralizedReilly}.

\begin{rem}
For minor technical reasons, it will be useful to record the following variants of the generalized Reilly formula, which were proved in \cite{KolesnikovEMilmanReillyPart1}:
\begin{itemize}
\item
If $u_\nu$ or $u$ are constant on $\partial M$ and $u \in \S_0(M)$, then:
\begin{multline}
\label{Reilly3}
\int_M (L u)^2 d\mu = \int_M \norm{\nabla^2 u}^2 d\mu + \int_M \scalar{ \Ric_\mu \; \nabla u, \nabla u} d\mu + \\
\int_{\partial M} H_\mu (u_\nu)^2 d\mu + \int_{\partial M} \scalar{\II_{\partial M}  \;\nabla_{\partial M} u,\nabla_{\partial M} u} d\mu ~.
\end{multline}
\item
If $u \in \S_D(M) := \S_0(M) \cap C^2(\partial M)$, then:
\begin{multline}
\label{Reilly2}
\int_M (L u)^2 d\mu = \int_M \norm{\nabla^2 u}^2 d\mu + \int_M \scalar{ \Ric_\mu \; \nabla u, \nabla u} d\mu + \\
\int_{\partial M} H_\mu (u_\nu)^2 d\mu + \int_{\partial M} \scalar{\II_{\partial M}  \;\nabla_{\partial M} u,\nabla_{\partial M} u} d\mu + 2 \int_{\partial M} u_\nu L_{\partial M} u \; d\mu_{\partial M}  ~.
\end{multline}
\end{itemize}
\end{rem}

\subsection{The $\CD(\rho,N)$ condition for $1/N \in [-\infty,1/n]$}

Given $u \in C^2_{loc}(M)$, denote:
\[
\Gamma_2(u) := \scalar{ \Ric_\mu \; \nabla u, \nabla u} + \norm{\nabla^2 u}^2.
\]
The following lemma for $1/N \in [0,1/n]$ is due to Bakry (e.g. \cite[Section 6]{BakryStFlour}); the extension to negative $N$ was noted in our companion work \cite{KolesnikovEMilmanReillyPart1}:

\begin{lem} \label{lem:CS}
For any $u \in C^2_{loc}(M)$ and $1/N \in [-\infty,1/n]$:
\begin{equation} \label{eq:Bakry-CS}
\Gamma_2(u) \geq \scalar{\Ric_{\mu,N}\; \nabla u, \nabla u} + \frac{1}{N} (Lu)^2 ~.
\end{equation}
Our convention throughout this work is that $-\infty \cdot 0 = 0$, and so if $Lu = 0$ at a point $p \in M$, the assertion when $\frac{1}{N} = -\infty$  is that:
\[
\Gamma_2(u) \geq \scalar{\Ric_{\mu,0}\; \nabla u, \nabla u} ~,
\]
at that point. 
\end{lem}

\begin{rem}
In fact, it is easy to show (e.g. \cite{BakryStFlour,KolesnikovEMilmanReillyPart1}) given $\rho \in \Real$ and $1/N \in (-\infty,1/n]$, that
$\Ric_{\mu,N} \geq \rho g$ on $M$ if and only if:
\[
\Gamma_2(u) \geq \rho \abs{\nabla u}^2 + \frac{1}{N} (Lu)^2 ~,~ \forall u \in C^2_{loc}(M) ~.
\]
\end{rem}

\subsection{Solution to Poisson Equation on Weighted Riemannian Manifolds}

As our manifold is smooth, connected, compact, with $C^2$ smooth boundary and strictly positive $C^2$-density all the way up to the boundary, all of the classical 
elliptic existence, uniqueness and regularity results (e.g. \cite[Chapter 8]{GilbargTrudinger}, \cite[Chapter 5]{LiebermanObliqueBook}, \cite[Chapter 3]{LadyEllipticBook}) 
immediately extend from the Euclidean setting to our weighted-manifold one (see e.g. \cite{Taylor-PDEBook-I,MorreyBook}); for more general situations (weaker regularity of metric, Lipschitz domains, etc.) see e.g. \cite{MitreaTaylor-PDEonLipManifolds} and the references therein. 
We summarize the results we require in the following:

\begin{thm}
Given a weighted-manifold $(M,g,\mu)$, $\mu = \exp(-V) d\vol_M$, we assume that $\partial M$ is $C^2$ smooth. Let $\alpha \in (0,1)$, and assume that $g$ is $C^{2,\alpha}$ smooth and $V \in C^{1,\alpha}(M)$. 
Let $f \in C^{\alpha}(M)$, $\varphi_D \in C^{2}(\partial M)$ and $\varphi_N \in C^{1}(\partial M)$. Then there exists a function $u \in C^{2,\alpha}_{loc}(int(M)) \cap C^{1,\beta}(M)$ for all $\beta \in (0,1)$, which solves: \[
L u = f ~ \text{on $M$} ~,
\]
with either of the following boundary conditions on $\partial M$:
\begin{enumerate}
\item Dirichlet: $u|_{\partial M} = \varphi_D$, assuming $\partial M \neq \emptyset$.  
\item Neumann: $u_\nu|_{\partial M} = \varphi_N$, assuming the following compatibility condition is satisfied:
\[
 \int_{M} f d\mu = \int_{\partial M} \varphi_N d\mu_{\partial M} ~.
\]
\end{enumerate}
In particular, $u \in \S_0(M)$ in either case. Moreover, $u \in \S_N(M)$ in the Neumann case and $u \in \S_D(M)$ in the Dirichlet case. 

\end{thm}

\begin{rem} \label{rem:C3}
As explained in \cite{KolesnikovEMilmanReillyPart1}, the generalized Reilly formula remains valid if the metric $g$ is only assumed to be $C^3$ smooth, in which case the above regularity results still apply. 
\end{rem}

We will not require the uniqueness of $u$ above, but for completeness we mention that this is indeed the case for Dirichlet boundary conditions, and up to an additive constant in the Neumann case. 

\subsection{Spectral-gap on the boundary of a weighted Riemannian manifold}

It is well-known (e.g. \cite{Taylor-PDE-II-Book}) that the symmetric operator $-L_{\partial M}$ on $L^2(\mu_{\partial M})$ with domain $C^1(\partial M)$ admits a (unique) self-adjoint positive semi-definite extension having discrete non-negative spectrum with corresponding complete orthonormal bases of eigenfunctions. The  best constant $\lambda_1$ in the following Poincar\'e inequality:
\[
\lambda_1 \Var_{\mu_{\partial M}}(f) \leq \int_{\partial M} \abs{\nabla_{\partial M} f}^2 d\mu_{\partial M} ~,~ \forall f \in C^1(\partial M) ~,
\]
then coincides with the spectral-gap, i.e. the first positive eigenvalue of $-L_{\partial M}$ away from the trivial zero eigenvalue corresponding to constant functions:
\[
-L_{\partial M} u_1 = \lambda_1 u_1 \text{ on $\partial M$ } ~.
\]
Since $\partial M$ is $C^2$ smooth and hence $V \in C^2(\partial M)$, elliptic regularity guarantees that all the eigenfunctions (and in particular $u_1$) are as smooth as $\partial M$, namely $C^2$ smooth.

\section{Poincar\'e-type inequalities on $\partial M$} \label{sec:Col}

\subsection{Generalized Colesanti Inequality}

Most of the results in this work are based on the following Poincar\'e-type inequality on $\partial M$, which appears to be novel in the Riemannian setting (even in the classical non-weighted setting, i.e. $N=n$):

\begin{thm}[Generalized Colesanti Inequality] \label{thm:Colesanti}
Assume that $(M,g,\mu)$ satisfies the $\CD(0,N)$ condition ($1/N \in [-\infty,1/n]$) and that $\II_{\partial M} > 0$ ($M$ is strictly locally-convex).
Then the following inequality holds for any $f \in C^{1}(\partial M)$:
\begin{equation} \label{eq:gen-full0}
 \int_{\partial M} H_\mu f^2 d\mu_{\partial M} - \frac{N-1}{N}\frac{\brac{\int _{\partial M} f d\mu_{\partial M}}^2}{\mu(M)} \leq  \int_{\partial M} \scalar{\II_{\partial M}^{-1} \;\nabla_{\partial M} f,\nabla_{\partial M} f} d\mu_{\partial M} ~.
\end{equation} 
\end{thm}

\begin{rem}
Thanks to our convention from Lemma \ref{lem:CS} that $-\infty \cdot 0 = 0$, observe that Theorem \ref{thm:Colesanti} is still meaningful when $N=0$ ($1/N = -\infty$), but only for functions $f$ with $\int _{\partial M} f d\mu_{\partial M} = 0$. 
\end{rem}

\begin{rem}
Theorem \ref{thm:Colesanti} was obtained by A.~Colesanti in \cite{ColesantiPoincareInequality} with $N=n$ for a compact subset $M$ of \emph{Euclidean space} $\Real^n$ having a $C^2$ strictly convex boundary and endowed with the Lebesgue measure ($V=0$). Colesanti was mainly interested in the case that $f$ has zero mean $\int_{\partial M} f \ d \mu_{\partial M}=0$, but his proof yields the additional second term in (\ref{eq:gen-full0}). This second term is of crucial importance, revealing the dependence on the generalized-dimension $N$, and demonstrating that even in the Euclidean case, Colesanti's original version remains valid for zero-mean functions under the weakest $CD(0,0)$ assumption. Colesanti derived this inequality as an infinitesimal version of the (purely Euclidean) Brunn-Minkowski inequality, and so his method is naturally \emph{confined to the Euclidean setting}; see \cite{ColesantiEugenia-PoincareFromAF} for further possible (Euclidean) extensions. As observed in \cite{ColesantiPoincareInequality}, Theorem \ref{thm:Colesanti} yields a sharp Poincar\'e inequality on $S^{n-1}$ when $M$ is a Euclidean ball in $\Real^n$. 
\end{rem}

Our proof of Theorem \ref{thm:Colesanti} is based on an $L^2(\mu)$-duality argument coupled with an application of the generalized Reilly formula. Further applications of $L^2$-duality in the context of obtaining Poincar\'e-type inequalities may be found in \cite{Hormander1965L2EstimatesAndDBarProblem,Helffer-DecayOfCorrelationsViaWittenLaplacian,LedouxSpinSystemsRevisited, KlartagUnconditionalVariance,BartheCorderoVariance,KlartagMomentMap}.  

\begin{proof}[Proof of Theorem \ref{thm:Colesanti}]
By applying the Cauchy--Schwarz inequality to the last-term in the generalized Reilly formula:
\[
2 \scalar{\nabla_{\partial M} u_\nu, \nabla_{\partial M} u} \leq \scalar{\II_{\partial M} \; \nabla_{\partial M} u , \nabla_{\partial M} u} + \scalar{\II_{\partial M}^{-1} \nabla_{\partial M} u_\nu , \nabla_{\partial M} u_\nu} ~,
\]
we obtain for any $u \in \S_N(M)$:
\begin{align*}
\int_M (L u)^2 d\mu  & \geq \int_M \brac{\norm{\nabla^2 u}^2 + \scalar{ \Ric_\mu \; \nabla u, \nabla u}} d\mu \\
& + \int_{\partial M} H_\mu (u_\nu)^2 d\mu_{\partial M} - \int_{\partial M} \scalar{\II_{\partial M}^{-1} \;\nabla_{\partial M} u_\nu,\nabla_{\partial M} u_\nu} d\mu_{\partial M} ~.
\end{align*}
Using the $\CD(0,N)$ condition as in Lemma \ref{lem:CS} with the convention that $-\infty \cdot 0 = 0$, we conclude:
\[
\frac{N-1}{N} \int_M (L u)^2 d\mu  \geq  \int_{\partial M} H_\mu (u_\nu)^2 d\mu_{\partial M} - \int_{\partial M} \scalar{ \II_{\partial M}^{-1} \;\nabla_{\partial M} u_\nu,\nabla_{\partial M} u_\nu} d\mu_{\partial M} ~.
\]

Given $f \in C^1(\partial M)$, we now
solve the following Neumann Laplace problem for $u \in \S_N(M)$ satisfying:
\[
L u \equiv \frac{1}{\mu(M)} \int_{\partial M} f d\mu_{\partial M}  \text{ on $M$ } ~,~  u_\nu = f  \text{ on $\partial M$} ~;
\]
note that the compatibility condition $\int_{\partial M} u_\nu d\mu_{\partial M} = \int_M (L u) d\mu$ is indeed satisfied, so that a solution exists. 
Plugging this back into the previous estimate, the asserted inequality (\ref{eq:gen-full0}) immediately follows. 
\end{proof}

\begin{rem}
Peculiarly, it is possible to strengthen this inequality when $N \neq 0$ by using it for $f+z$ and optimizing over $z \in \Real$; alternatively and equivalently, we may solve in the last step above:
\[
L u \equiv z \text{ on $M$ } ~,~  u_\nu = f - \dashint_{\partial M} f d\mu_{\partial M}  +  z \frac{\mu(M)}{\mu_{\partial M}(\partial M)} \text{ on $\partial M$} ~.
\]
This results in the following stronger inequality: \[
\int_{\partial M} H_\mu f^2 d\mu_{\partial M} - \frac{N-1}{N}\frac{\brac{\int _{\partial M} f d\mu_{\partial M}}^2}{\mu(M)} + \frac{\brac{ \int_{\partial M} f \beta d\mu_{\partial M} }^2}{\int_{\partial M} \beta d\mu_{\partial M}} \leq  \int_{\partial M} \scalar{\II_{\partial M}^{-1} \;\nabla_{\partial M} f,\nabla_{\partial M} f} d\mu_{\partial M} ~,
\]
where:
\[
\beta(x) := \frac{N-1}{N} \frac{\mu_{\partial M}(\partial M)}{\mu(M)} - H_\mu(x) ~.
\]
Note that $\int_{\partial M} \beta d\mu_{\partial M} \geq 0$ by testing (\ref{eq:gen-full0}) on the constant function $f \equiv 1$. It may be shown that this integral is in fact strictly positive, unless $M$ is isometric to a Euclidean ball and $V$ is constant - see Remark \ref{rem:HR-equality}; so in all other cases, this yields a strict improvement over (\ref{eq:gen-full0}). 

\noindent
By Colesanti's argument in the Euclidean setting, the weaker (\ref{eq:gen-full0}) inequality constitutes an infinitesimal version of the (sharp) Brunn--Minkowski inequality (for convex sets), and so one cannot hope to improve (\ref{eq:gen-full0}) in the corresponding cases where Brunn--Minkowski is sharp. 
On the other hand, it would be interesting to integrate back the stronger inequality and obtain a refined version of Brunn--Minkowski, which would perhaps be better suited for obtaining delicate stability results. 
\end{rem}

\subsection{A Dual Version}

Next, we establish a dual version of Theorem \ref{thm:Colesanti}, which in fact applies whenever $M$ is only assumed (generalized) \emph{mean-convex} and under a general $\CD(\rho,N)$ condition; however, the price we pay is that we do not witness the dependence on $N$ in the resulting inequality, so we might as well assume $\CD(\rho,0)$.

\begin{thm}[Dual Generalized Colesanti Inequality] \label{thm:dual-Colesanti}
Assume that $(M,g,\mu)$ satisfies the $\CD(\rho,0)$ condition, $\rho \in \Real$, and that $H_\mu > 0$ on $\partial M$ ($M$ is strictly generalized mean-convex). Then for any $f \in C^{2}(\partial M)$ and $C \in \Real$:
\[
\int_{\partial M} \scalar{\II_{\partial M}  \;\nabla_{\partial M} f,\nabla_{\partial M} f} d\mu_{\partial M} \leq \int_{\partial M} \frac{1}{H_\mu} \Bigl(L_{\partial M} f + \frac{\rho (f-C) }{2} \Bigr)^2 d\mu_{\partial M} ~.
\]
\end{thm}
\begin{proof}
This time, we solve the Dirichlet Laplace problem for $u \in \S_D(M)$ satisfying:
\[
Lu \equiv 0 \text{ on $M$ } ~,~ u = f \text{ on $\partial M$} ~.
\]
By the generalized Reilly formula (as in (\ref{Reilly2})) and the $\CD(\rho,0)$ condition: \begin{align*}
&  0 \geq \rho \int_{M} \abs{\nabla u}^2 d \mu \\
& + \int_{\partial M} H_\mu (u_\nu)^2  d \mu_{\partial M} + \int_{\partial M} \scalar{ \II_{\partial M} \; \nabla_{\partial M} f, \nabla_{\partial M} f }  d \mu_{\partial M}  + 2 \int_{\partial M} u_\nu  L_{\partial M} f  \;  d \mu_{\partial M} ~.
\end{align*}
Integrating by parts we obtain:
\begin{align*}
& 0 \geq \rho \int_{\partial M} f u_\nu d \mu_{\partial M}\\
 & + \int_{\partial M} H_\mu (u_\nu)^2  d \mu_{\partial M} + \int_{\partial M} \scalar{ \II_{\partial M} \; \nabla_{\partial M} f, \nabla_{\partial M} f }  d \mu_{\partial M} + 2 \int_{\partial M} u_\nu  L_{\partial M} f  \;  d \mu_{\partial M} ~.
\end{align*}
Since $\int_{\partial M} u_\nu d\mu_{\partial M} = \int_M (L u) d\mu = 0$, we may as well replace the first term above by $\int_{\partial M} (f - C) u_\nu  d \mu_{\partial M}$. 
The asserted inequality is obtained following an application of the Cauchy--Schwarz inequality:
\[
H_\mu u^2_\nu + 2 u_\nu \Bigl( L_{\partial M} f +  \frac{\rho (f-C)}{2} \Bigr) \ge -  \frac{1}{H_\mu} \Bigl(L_{\partial M} f + \frac{\rho (f-C)}{2} \Bigr)^2 ~.
\]
\end{proof}

\begin{rem}
When $\II_{\partial M} > 0$ and $\rho=0$, Theorem \ref{thm:dual-Colesanti} for the $\CD(0,\infty)$ condition may be heuristically obtained from Theorem \ref{thm:Colesanti} by a formal non-rigorous duality argument:
\begin{multline*}
\int_{\partial M} \scalar{\II_{\partial M} \; \nabla_{\partial M} f , \nabla_{\partial M} f} d\mu_{\partial M}  =^?  \sup_{g} \frac{\brac{\int_{\partial M} \scalar{\nabla_{\partial M} f,\nabla_{\partial M} g} d\mu_{\partial M} }^2 }{ \int_{\partial M} \scalar{\II_{\partial M}^{-1} \; \nabla_{\partial M} g, \nabla_{\partial M} g} d\mu_{\partial M}} \\
\leq
\sup_{g} \frac{\brac{\int_{\partial M} g L_{\partial M} f d\mu_{\partial M} }^2 }{ \int_{\partial M} H_\mu g^2 d\mu_{\partial M}} \leq \int_{\partial M} \frac{1}{H_\mu} (L_{\partial M} f)^2 d\mu_{\partial M} ~,
\end{multline*}
where the supremum above is over all functions $g \in C^{1}(\partial M)$ with $\int_{\partial M} g d\mu_{\partial M} = 0$. 
The delicate point is justifying the equality under the question-mark above: Cauchy--Schwarz implies the $\geq$ direction, and so given $f \in C^{1}(\partial M)$ it remains to find a function $g \in C^{1}(\partial M)$ so that $\nabla_{\partial M} g = \II_{\partial M} \nabla_{\partial M} f$ on $\partial M$. It is well known that on a simply-connected manifold (and more generally, with vanishing first homology), a vector field is a gradient field if and only if its covariant derivative is a symmetric tensor, but this does not seem to be the case for us. 
\end{rem}

\section{Applications of Generalized Colesanti Inequalities} \label{sec:Col-App}

\subsection{Topological Consequences}

\begin{thm}
Assume that $(M,g,\mu)$ satisfies the $\CD(0,0)$ condition and that $\II_{\partial M} > 0$ ($M$ is strictly locally-convex). 
Then $\partial M$ is connected. 
\end{thm}
\begin{proof}
Otherwise, $\partial M$ has at least two connected components. By constructing a function $f \in C^1(\partial M)$ which is equal to an appropriate constant on each of the components so that $\int_{\partial M} f d\mu_{\partial M} = 0$, we obtain a contradiction to (\ref{eq:gen-full0}). 
\end{proof}
\begin{rem}
Observe that one cannot relax most of the conditions of the theorem. For instance, taking $M$ to be $[0,1] \times T^{n-1}$ with the product metric, where $T^{n-1}$ is the flat $n-1$-dimensional torus, we see that the strict convexity condition cannot be relaxed to $\II_{\partial M}\geq 0$. In addition, taking $M$ to be the submanifold of Hyperbolic space $H$, which in the Poincar\'e model in the open unit-disc in $\Real^n$ is represented by:
\[
M := \set{ x \in \Real^n \; ; \; \abs{x} < 1 ~,~ \abs{x + 10 e_n} < 10.5 ~,~ \abs{x - 10 e_n} < 10.5 } ~,
\]
since $M$ is strictly convex as a subset of Euclidean space, the same holds in $H$, but $\partial M$ has two connected components. Consequently, we see that the $\CD(0,0)$ condition cannot be relaxed to $\CD(-(n-1),0)$ and hence (by scaling the metric) neither to $\CD(-\eps,0)$.  
\end{rem}

\subsection{Mean-Curvature Inequalities}

Setting $f\equiv 1$ in Theorem \ref{thm:Colesanti}, we recover and generalize to the entire range $1/N \in [-\infty,1/n]$ the following recent result of Huang and Ruan \cite[Theorem 1.3]{HuangRuanMeanCurvatureEstimates} for $N \in [n,\infty]$, who generalized the same result obtained by Reilly \cite{ReillyMeanCurvatureEstimate} in the classical Riemannian volume case ($V=0$ and $N=n$).

\begin{cor}[Extending Reilly and Huang--Ruan] \label{cor:HuangRuan1}
Assume that $(M,g,\mu)$ satisfies the $\CD(0,N)$ condition ($1/N \in (-\infty,1/n]$) and that $\II_{\partial M} > 0$ ($M$ is strictly locally-convex). Then:
\begin{equation} \label{eq:HR-1}
\int_{\partial M} H_\mu d\mu_{\partial M} \leq \frac{N-1}{N} \frac{\mu_{\partial M}(\partial M)^2}{\mu(M)} ~.
\end{equation} 
\end{cor}

Applying Cauchy--Schwarz, it immediately follows that in above setting:
\begin{equation} \label{eq:int-1-H} 
\int_{\partial M} \frac{1}{H_\mu} d\mu_{\partial M} \geq \frac{\mu_{\partial M}(\partial M)^2}{\int_{\partial M} H_\mu d\mu_{\partial M}} \geq \frac{N}{N-1} \mu(M) ~.
\end{equation}
Interestingly, it was shown by A. Ros \cite{RosMeanCurvatureEstimateAndApplication} in the classical non-weighted case, and generalized by Huang and Ruan \cite[Theorem 1.1]{HuangRuanMeanCurvatureEstimates} to the weighted-Riemannian setting for $N \in [n,\infty]$, that it is enough to assume that $M$ is strictly (generalized) mean-convex for the inequality between first and last terms in (\ref{eq:int-1-H}) to hold. We extend this to the entire range $1/N \in (-\infty,1/n]$:
\begin{thm}[Extending Ros and Huang--Ruan] \label{thm:HuangRuan2}
Assume that $(M,g,\mu)$ satisfies the $\CD(0,N)$ condition ($1/N \in (-\infty,1/n]$) and that $H_{\mu} > 0$ ($M$ is strictly generalized mean-convex). Then:
\begin{equation} \label{eq:HR-2}
\int_{\partial M} \frac{1}{H_\mu} d\mu_{\partial M} \geq \frac{N}{N-1} \mu(M) ~.
\end{equation}
\end{thm}
 
This is very much related to our dual version of the generalized Colesanti inequality (Theorem \ref{thm:dual-Colesanti}), and in fact both inequalities may be obtained simultaneously from the generalized Reilly formula by invoking the Cauchy--Schwarz inequality in two different ways. In a sense, this explains why we lost the dependence on $N$ in Theorem \ref{thm:dual-Colesanti} and why we lose the dependence on $\rho$ in Theorem \ref{thm:HuangRuan2}. 
The idea for proving Theorem \ref{thm:HuangRuan2} is the same as in \cite{HuangRuanMeanCurvatureEstimates}, but our argument is somewhat more direct.
\begin{proof}[Proof of Theorem \ref{thm:HuangRuan2}]
Let us solve for $u \in \S_0(M)$ the following Dirichlet Poisson equation:
\[
Lu \equiv 1 \text{ on $M$ } ~,~ u \equiv 0\text{ on $\partial M$} ~.
\]
By the generalized Reilly formula (as in (\ref{Reilly3})) and the $\CD(0,N)$ condition: 
\begin{eqnarray*}
\mu(M) = \int_M (Lu)^2 d\mu &=& \int_M \norm{\nabla^2 u}^2 d\mu + \int_M \scalar{\Ric_\mu \; \nabla u, \nabla u} d\mu + \int_{\partial M} H_\mu (u_\nu)^2  d \mu_{\partial M} \\
&\geq & \frac{1}{N} \int_M (Lu)^2 d\mu + \int_{\partial M} H_\mu (u_\nu)^2  d \mu_{\partial M} ~.
\end{eqnarray*}
Coupled with an application of the Cauchy--Schwarz inequality, this yields:
\begin{align*}
\mu(M)^2 & = (\int_M (L u) d\mu)^2 = (\int_{\partial M} u_\nu d\mu_{\partial M})^2 \\
& \leq \int_{\partial M} H_\mu (u_\nu)^2 d\mu_{\partial M} \int_{\partial M} \frac{1}{H_\mu} d\mu_{\partial M} \leq \frac{N-1}{N} \mu(M) \int_{\partial M} \frac{1}{H_\mu} d\mu_{\partial M} ~,
\end{align*}
and the assertion follows. 
\end{proof}

\begin{rem} \label{rem:HR-equality}
It may be shown by analyzing the cases of equality in all of the above used inequalities, that when $N \in [n,\infty]$, 
equality occurs in (\ref{eq:HR-1}) or (\ref{eq:HR-2}) if and only if $M$ is isometric to a Euclidean ball and $V$ is constant. See \cite{ReillyMeanCurvatureEstimate,RosMeanCurvatureEstimateAndApplication,HuangRuanMeanCurvatureEstimates} for more details. 
\end{rem}

\subsection{Spectral-Gap Estimates on $\partial M$}

Next, we recall a result of Xia \cite{XiaSpectralGapOnConvexBoundary} in the classical non-weighted Riemannian setting ($V=0$), stating that when $\Ric_g \geq 0$ on $M$ and $\II_{\partial M}\geq \sigma g|_{\partial M}$ on $\partial M$ with $\sigma > 0$, then:
\begin{equation} \label{eq:Xia}
\Var_{\vol_{\partial M}}(f) \leq \frac{1}{(n-1) \sigma^2} \int_{\partial M} |\nabla_{\partial M} f|^2 d\vol_{\partial M} ~,~ \forall f \in C^1(\partial M) ~.
\end{equation}
In other words, the spectral-gap of $-L_{\partial M}$ on $(\partial M,g|_{\partial M},\vol_{\partial M})$ away from the trivial zero eigenvalue is at least $(n-1) \sigma^2$. 
Since in that case we have $H_g = tr(\II_{\partial M})\geq (n-1) \sigma$, our next result, which is an immediate corollary of Theorem \ref{thm:Colesanti} applied to $f$ with $\int_{\partial M} f \ d \mu_{\partial M}=0$, is both a refinement and an extension of Xia's estimate to the more general $\CD(0,0)$ condition in the weighted Riemannian setting:
\begin{cor} \label{cor:IIH-Poincare}
Assume that $(M,g,\mu)$ satisfies $\CD(0,0)$, and that $\II_{\partial M} \ge \sigma g|_{\partial M}$, $H_\mu \ge \xi$ on $\partial M$ with $\sigma,\xi > 0$. Then:
\[\Var_{\mu_{\partial M}}(f) \leq \frac{1}{\sigma \xi} \int_{\partial M} |\nabla_{\partial M} f|^2 d \mu_{\partial M} ~,~ \forall f \in C^1(\partial M) ~.
\]\end{cor}

In the Euclidean setting, and more generally when all sectional curvatures are non-negative, an improved bound will be obtained in the next section. 
The next result extends the previous one to the $\CD(\rho,0)$ setting:

\begin{cor} \label{cor:IIHRho-Poincare}
Assume that $(M,g,\mu)$ satisfies $\CD(\rho,0)$, $\rho \geq 0$, that $\partial M$ is $C^{2,\alpha}$ smooth and that $\II_{\partial M} \ge \sigma g|_{\partial M}$, $H_\mu \ge \xi$ on $\partial M$ with $\sigma,\xi > 0$. Then:
\[\lambda_1 \Var_{\mu_{\partial M}}(f) \leq \int_{\partial M} |\nabla_{\partial M} f|^2 d \mu_{\partial M} ~,~ \forall f \in C^1(\partial M) ~,
\]with:
\[
\lambda_1 \geq \frac{\rho + a + \sqrt{2 a \rho + a^2}}{2} \geq \max\brac{a, \frac{\rho}{2}}  ~,~ a := \sigma \xi ~.
\]
\end{cor}
\begin{proof}
Let $u \in C^2(\partial M)$ denote the first non-trivial eigenfunction of $-L_{\partial M}$, satisfying $-L_{\partial M} u = \lambda_1 u$ with $\lambda_1 > 0$ the spectral-gap  (we already know it is positive by Corollary \ref{cor:IIH-Poincare}). 
Plugging the estimates $\II_{\partial M} \ge \sigma g|_{\partial M}$, $H_\mu \ge \xi$ into the dual generalized Colesanti inequality (Theorem \ref{thm:dual-Colesanti}) and applying it to the function $u$, we obtain:
\[
\sigma \lambda_1 \int_{\partial M} u^2 d\mu_{\partial M} \leq \frac{1}{\xi} \int_{\partial M} (- \lambda_1 u  + \frac{\rho_1}{2} u)^2 d\mu_{\partial M} ~,~ \forall \rho_1 \in [0,\rho] ~.
\]
Opening the brackets, this yields:
\[
\lambda_1^2 -  (\rho_1 + \xi \sigma) \lambda_1 + \frac{\rho_1^2}{4} \geq 0 ~,~ \forall \rho_1 \in [0,\rho] ~.
\] 
The assertion then follows by using all values of $\rho_1 \in [0,\rho]$.

\end{proof}

In the next section, we extend our spectral-gap estimates on $(\partial M,g|_{\partial M},\mu_{\partial M})$ for the case of varying lower bounds $\sigma$ and $\xi$.

\section{Boundaries of $\CD(\rho,N)$ weighted-manifolds} \label{sec:boundaries}

Throughout this section, we assume that $n \geq 3$. Recall that $\partial M$ is assumed to be $C^2$ smooth, and consequently so is the induced Riemannian metric on $\partial M$. 

\subsection{Curvature-Dimension of the Boundary}

Denote the full Riemann curvature $4$-tensor on $(M,g)$ by $\rm{R}^M_g$, and let $\Ric^{\partial M}_{\mu_{\partial M}}$ denote the weighted Ricci tensor on $(\partial M,g|_{\partial M}, \mu_{\partial M})$. 

\begin{lem} \label{lem:boundary-Ric}
Set $g_0 := g|_{\partial M}$ the induced metric on $\partial M$. Then:
\[
\Ric^{\partial M}_{\mu_{\partial M}} = (\Ric^M_{\mu} - \rm{R}^M_g(\cdot,\nu,\cdot,\nu))|_{T \partial M}+ (H_\mu g_0 - \II_{\partial M}) \II_{\partial M} ~.
\]
\end{lem}
\begin{proof}
Let $e_1,\ldots,e_n$ denote an orthonormal frame of vector fields in $M$ so that $e_n$ coincides on $\partial M$ with the outer normal $\nu$. The Gauss formula asserts that for any $i,j,k,l \in \set{1,\ldots,n-1}$:
\[
\rm{R}^{\partial M}_{g_0}(e_i,e_j,e_k,e_l) = \rm{R}^{M}_{g}(e_i,e_j,e_k,e_l) + \II_{\partial M}(e_i,e_k) \II_{\partial M}(e_j,e_l) - \II_{\partial M}(e_j,e_k) \II_{\partial M}(e_i,e_l)  ~.
\]
Contracting by applying $g_0^{j,l}$ and using the orthogonality, we obtain:
\begin{equation} \label{eq:calc1}
\Ric^{\partial M}_{g_0} = (\Ric^M_g - \rm{R}^M_g(\cdot,\nu,\cdot,\nu))|_{T \partial M} + (H_g g_0 - \II_{\partial M}) \II_{\partial M} ~.
\end{equation}
In addition we have:
\begin{eqnarray*}
\nabla^2_{g_0} V(e_i,e_j) &=& e_i (e_j(V)) - ((\nabla_{\partial M})_{e_i} e_j)(V) \\
&= & e_i (e_j(V)) - ((\nabla_{M})_{e_i} e_j)(V) - \II_{\partial M}(e_i,e_j) e_n(V) \\
& = & \nabla^2_g V(e_i,e_j) - \II_{\partial M}(e_i,e_j) \; g(\nabla V , \nu) ~.
\end{eqnarray*}
In other words:
\begin{equation} \label{eq:calc2}
\nabla^2_{g_0} V = \nabla^2_g V|_{T \partial M} - \II_{\partial M} g(\nabla V, \nu) ~.
\end{equation}
Adding (\ref{eq:calc1}) and (\ref{eq:calc2}) and using that $H_\mu = H_g - g(\nabla V,\nu)$, the assertion follows. 
\end{proof}

\begin{cor} \label{cor:boundary-CD}
Assume that $0 \leq \II_{\partial M} \leq H_\mu g_0$ and $\rm{R}^M_g(\cdot,\nu,\cdot,\nu) \leq \kappa g_0$ as $2$-tensors on $\partial M$. If $(M,g,\mu)$ satisfies $\CD(\rho,N)$ then $(\partial M ,g_0 ,\mu_{\partial M})$ satisfies $\CD(\rho-\kappa,N-1)$. 
\end{cor}
\begin{proof}
The first assumption ensures that:
\[
(H_\mu g_0 - \II_{\partial M}) \II_{\partial M} \geq 0 ~,
\]
since the product of two commuting positive semi-definite matrices is itself positive semi-definite. 
It follows by Lemma \ref{lem:boundary-Ric} that:
\begin{eqnarray*}
& & \Ric^{\partial M}_{\mu_{\partial M},N-1} = \Ric^{\partial M}_{\mu_{\partial M}} - \frac{1}{N-1-(n-1)} dV \otimes dV |_{T \partial M}\\
&\geq & (\Ric^M_\mu - \rm{R}^M_g(\cdot,\nu,\cdot,\nu) - \frac{1}{N-n} dV \otimes dV) |_{T \partial M} = (\Ric^M_{\mu,N} - \rm{R}^M_g(\cdot,\nu,\cdot,\nu))|_{T \partial M} ~.
\end{eqnarray*}
The assertion follows from our assumption that $\Ric^M_{\mu,N} \geq \rho g$ and $\rm{R}^M_g(\cdot,\nu,\cdot,\nu)|_{T \partial M} \leq \kappa g_0$. 
\end{proof}

An immediate modification of the above argument yields:
\begin{cor}
Assume that $(M,g,\mu)$ satisfies $\CD(\rho,N)$ and that $\rm{R}^M_g(\cdot,\nu,\cdot,\nu) \leq \kappa g_0$ as $2$-tensors on $\partial M$. If $\sigma_1 g_0 \leq \II_{\partial M} \leq \sigma_2 g_0$, for some functions $\sigma_1,\sigma_2 : \partial M \rightarrow \Real$, then:
\[
\Ric^{\partial M}_{\mu_{\partial M},N-1} \geq (\rho -\kappa + \min(\sigma_1(H_\mu -\sigma_1) , \sigma_2 (H_\mu -\sigma_2))) g_0 ~.
\] 
In particular, if $H_\mu \geq \xi $ and $\sigma g_0 \leq \II_{\partial M} \leq (H_\mu -\sigma) g_0$ for some constants $\xi,\sigma \in \Real$, then:
\[
\Ric^{\partial M}_{\mu_{\partial M},N-1} \geq (\rho - \kappa + \sigma (\xi - \sigma)) g_0 ~.
\]
\end{cor}

When $\II_{\partial M} \geq \sigma g_0$ with $\sigma \geq 0$, we obviously have $\II_{\partial M} \leq (H_g - (n-2) \sigma) g_0$ and $H_g \geq (n-1) \sigma$.
In addition if $\scalar{\nabla V,\nu} \leq 0$ on $\partial M$, we have $H_g \leq H_\mu$. Consequently, we obtain:
\begin{equation} \label{eq:boundary-Ric-estimate}
(H_\mu g_0 - \II_{\partial M}) \II_{\partial M} \geq \sigma (H_g - \sigma) g_0 \geq (n-2) \sigma^2 g_0 ~.
\end{equation}
This is summarized in the following:
\begin{prop} \label{prop:CD-boundary}
Assume that $(M^n,g,\mu)$ satisfies $\CD(\rho,N)$, $\II_{\partial M} \geq \sigma g_0$ with $\sigma \geq 0$, $\scalar{\nabla V,\nu} \leq 0$ and $R^M_g(\cdot,\nu,\cdot,\nu) \leq \kappa g_0$ on $T \partial M$.
Then $(\partial M,g_0,\mu_{\partial M})$ satisfies $\CD(\rho_0,N-1)$ with:
\[
\rho_0 = \rho - \kappa +(n-2)\sigma^2 .
\] 
Moreover, if $H_g \geq \xi \; (\geq (n-1) \sigma)$, one may in fact use:
\[
\rho_0 = \rho - \kappa + \sigma (\xi - \sigma) .
\]
\end{prop}

Observe that this is sharp for the sphere $S^{n-1}$, both as a hypersurface of Euclidean space $\Real^n$, and as a hypersurface in a sphere $R S^{n}$ with radius $R \geq 1$. 

\medskip

We can now apply the known results for weighted-manifolds (without boundary!) satisfying the $\CD(\rho_0,N-1)$ condition to $(\partial M, g_0,\mu_{\partial M})$.

\subsection{log-Sobolev Estimates on $\partial M$}

\medskip

The first estimate is an immediate consequence of the Bakry--\'Emery criterion \cite{BakryEmery} for log-Sobolev inequalities, see \cite{Ledoux-Book} for definitions and more details:
\begin{cor}  \label{cor:LS}
With the same assumptions and notation as in Proposition \ref{prop:CD-boundary} and for $N \in [n,\infty]$, $(\partial M,g_0,\mu_{\partial M})$ satisfies a log-Sobolev inequality with constant:
\[
\lambda_{LS} := \rho_0 \frac{N-1}{N-2} \; ,
\]
assuming that the latter quantity is positive. 
\end{cor}
\begin{rem} \label{rem:LS}
In particular, since a log-Sobolev inequality always implies a spectral-gap estimate \cite{Ledoux-Book}, this is a strengthening of the Lichnerowicz estimate  $\lambda_1 \geq \rho_0 \frac{N-1}{N-2}$ \cite{KolesnikovEMilmanReillyPart1}. 
\end{rem}

The latter yields a log-Sobolev strengthening over Xia's spectral-gap estimate (\ref{eq:Xia}) for the boundary of a strictly locally-convex manifold of non-negative sectional curvature. For concreteness, we illustrate this below for geodesic balls:

\begin{example}
Assume that $\partial M = \emptyset$, $n \geq 3$, and that $(M^n,g)$ has sectional curvatures in the interval $[\kappa_0,\kappa_1]$. Let $B_r$ denote a geodesic ball around $p \in M$ of radius $0 < r \leq \text{inj}_p$, where $\text{inj}_p$ denotes the radius of injectivity at $p$, and consider $(\partial B_r,g_0,\vol_{\partial B_r})$ where $g_0 = g|_{\partial B_r}$. 
By \cite[Chapter 6, Theorem 27]{PetersenBook2ndEd}, $\II_{\partial B_r} \geq \sqrt{\kappa_1} \cot(\sqrt{\kappa_1} r) g_0$. Consequently, by Lemma \ref{lem:boundary-Ric}:
\[
\Ric^{\partial B_r}_{g_0}\geq (n-2) \kappa_0 g_0 + (H_{g} g_0 - \II_{\partial B_r}) \II_{\partial B_r}  \geq \rho_0 g_0 ~,~ \rho_0 := (n-2) \brac{\kappa_0 + \kappa_1 \cot^2(\sqrt{\kappa_1} r)} ~.
\]
It follows that $(\partial B_r,g_0,\vol_{\partial B_r})$ satisfies $\CD(\rho_0,n-1)$, and hence by the Bakry--\'Emery criterion as above this manifold satisfies a log-Sobolev inequality with constant $\lambda_{LS} \geq \frac{n-1}{n-2} \rho_0 = (n-1) (\kappa_0 + \kappa_1 \cot^2(\sqrt{\kappa_1} r))$ whenever the latter is positive, strengthening Xia's result for the spectral-gap (\ref{eq:Xia}) in the case of non-negative sectional curvature ($\kappa_0 = 0$). 

Furthermore, if we replace the lower bound assumption on the sectional curvatures by the assumption that $\Ric^{B_r}_g \geq \rho g$, we obtain by Corollary \ref{cor:LS} that $\lambda_{LS} \geq (n-1) (\frac{\rho - \kappa_1}{n-2} + \kappa_1 \cot^2(\sqrt{\kappa_1} r))$ whenever the latter is positive. 
\end{example}

\subsection{Spectral-Gap Estimates on $\partial M$ involving varying curvature}

Proceeding onward, we formulate our next results in Euclidean space with constant density (satisfying $\CD(0,n)$), since then the assumptions of Proposition \ref{prop:CD-boundary} are the easiest to enforce. We continue to denote by $g_0$ the induced Euclidean metric on $\partial M$, which by assumption is guaranteed to be $C^2$ smooth. In fact, we may always assume it is as smooth as we like, since the general case of a $C^2$ smooth boundary follows by a standard Euclidean approximation argument. 
 
 \medskip
We first record the Lichnerowicz spectral-gap estimate already mentioned in Remark \ref{rem:LS} (cf. \cite{KolesnikovEMilmanReillyPart1}). It improves (in the Euclidean setting) the spectral-gap estimate given by Corollary \ref{cor:IIH-Poincare} (which applies to general $\CD(0,0)$ weighted-manifolds).

\begin{thm}[Lichnerowicz Estimate on $\partial M$] \label{thm:Lich-On-Boundary}
Let $M$ denote a compact subset of Euclidean space $(\Real^n,g)$ with $C^2$-smooth boundary. Assume that:
\[
\II_{\partial M} \geq \sigma g_0 ~,~ tr(\II_{\partial M}) \geq \xi ~,
\]
for some $\sigma,\xi > 0$. Then:
\[
\Var_{\vol_{\partial M}}(f) \leq \frac{n-2}{n-1} \frac{1}{(\xi-\sigma)\sigma} \int_{\partial M} \abs{\nabla_{\partial M} f}^2 d\vol_{\partial M} ~,~ \forall f \in C^1(\partial M) ~.
\]
\end{thm}

When the $\sigma,\xi > 0$ above are non-uniform, the next result follows from Proposition \ref{prop:CD-boundary} (or more precisely (\ref{eq:boundary-Ric-estimate})) coupled with a refinement of the Lichnerowicz estimate due to L.~Veysseire \cite{VeysseireSpectralGapEstimateCRAS} (cf. \cite{KolesnikovEMilmanReillyPart1}):

\begin{thm}[Veysseire Estimate on $\partial M$] \label{thm:Vey-On-Boundary}
Let $M$ denote a compact subset of Euclidean space $(\Real^n,g)$ with $C^2$-smooth boundary.
Assume that:
\[
\II_{\partial M} \geq \sigma g_0 ~,~ tr(\II_{\partial M}) \geq \xi ~,
\]
for some positive measurable functions $\sigma,\xi : \partial M \rightarrow \Real_+$. Then: 
\[
\Var_{\vol_{\partial M}}(f) \leq \dashint_{\partial M} \frac{1}{ (\xi -\sigma) \sigma} d\vol_{\partial M} \int_{\partial M} \abs{\nabla_{\partial M} f}^2 d\vol_{\partial M} ~,~ \forall f \in C^1(\partial M) ~.
\]
\end{thm}

Both of the above results may be extended to a weighted Riemannian setting as long as the assumptions of Proposition \ref{prop:CD-boundary} remain valid (and say the boundary is assumed to be $C^3$ smooth, cf. Remark \ref{rem:C3}), but this is not so easy to enforce. In contrast, we conclude this section by providing an alternative argument, based on our generalized Colesanti inequality, which is much easier to generalize to the Riemannian setting (see Remark \ref{rem:Col-On-Boundary} below). 

\begin{thm}[Colesanti Estimate on $\partial M$] \label{thm:Col-On-Boundary}
With the same assumptions as in the previous theorem, we have:
\[
\Var_{\vol_{\partial M}}(f) \leq C \brac{\dashint_{\partial M} \frac{1}{\xi} d\vol_{\partial M} \dashint_{\partial M} \frac{1}{\sigma} d\vol_{\partial M}} \int_{\partial M} \abs{\nabla_{\partial M} f}^2 d\vol_{\partial M} ~,~ \forall f \in C^1(\partial M) ~,
\]
where $C>1$ is some universal (dimension-independent) numeric constant.
\end{thm}

\begin{proof}
Given a $1$-Lipschitz function $f : \partial M \rightarrow \Real$ with $\int_{\partial M} f d\vol_{\partial M} = 0$, we may estimate using Cauchy--Schwarz and Theorem \ref{thm:Colesanti}:
\begin{eqnarray}
\nonumber \Bigl(\dashint_{\partial M} \abs{f} d\vol_{\partial M} \Bigr)^2 &\leq& \dashint_{\partial M} \frac{1}{\xi} d\vol_{\partial M}\dashint_{\partial M} \xi f^2 d\vol_{\partial M} \\
\nonumber &\leq&  \dashint_{\partial M} \frac{1}{\xi} d\vol_{\partial M}  \dashint_{\partial M} \scalar{\II_{\partial M}^{-1} \;\nabla_{\partial M} f,\nabla_{\partial M} f} d\vol_{\partial M} \\
\label{eq:FM-estimate} & \leq & \dashint_{\partial M} \frac{1}{\xi} d\vol_{\partial M} \dashint_{\partial M} \frac{1}{\sigma} d\vol_{\partial M} ~.
\end{eqnarray}
It follows by a general result of the second-named author \cite{EMilman-RoleOfConvexity}, which applies to any weighted-manifold satisfying the $\CD(0,\infty)$ condition, and in particular to $(\partial M, g_0 , \vol_{\partial M})$, that up to a universal constant, the same estimate as in (\ref{eq:FM-estimate}) holds for the variance of any function $f \in C^{1}(\partial M)$ with $\dashint_{\partial M} \abs{\nabla_{\partial M} f}^2 d\vol_{\partial M} \leq 1$, and so the assertion follows. \end{proof}

Note that the estimate given by Theorem \ref{thm:Lich-On-Boundary} is sharp and that the ones in Theorems \ref{thm:Vey-On-Boundary} and \ref{thm:Col-On-Boundary} are sharp up to constants, as witnessed by $S^{n-1} \subset \Real^n$. 

\begin{rem} \label{rem:Col-On-Boundary}
In Theorem \ref{thm:Col-On-Boundary}, the Euclidean setting was only used to establish that $(\partial M, g_0, \vol_{\partial M})$ satisfies $\CD(0,\infty)$. 
It is immediate to check that Theorem \ref{thm:Col-On-Boundary} generalizes to the formulation given in Theorem \ref{thm:intro-Col-On-Boundary} from the Introduction, whenever $(M, g, \mu)$ satisfies $\CD(0,0)$ (owing to our convention that $-\infty \cdot 0 = 0$) and $(\partial M,g_0,\mu_{\partial M})$ satisfies $\CD(0,\infty)$.
We remark that the results from \cite{EMilman-RoleOfConvexity} used in the proof of Theorem \ref{thm:Col-On-Boundary} were obtained assuming the metric in question is $C^\infty$ smooth, but an inspection of the proof, which builds upon the regularity results in \cite{MorganRegularityOfMinimizers}, verifies that it is enough to have a $C^2$ metric. 
\end{rem}

\section{Connections to the  Brunn--Minkowski Theory} \label{sec:BM}

It was shown by Colesanti \cite{ColesantiPoincareInequality} that in the Euclidean case, the inequality (\ref{eq:gen-full0}) is \emph{equivalent} to the statement that the function $t \mapsto \vol(K + t L)^{1/n}$ is concave at $t=0$ when $K,L$ are strictly convex and $C^2$ smooth. Here $A + B := \set{a + b \; ; \; a \in A , b \in B}$ denotes Minkowski addition. Using homogeneity of the volume and a standard approximation procedure of arbitrary convex sets by ones as above, this is in turn equivalent to:
\[
\vol((1-t) K + t L)^{1/n} \geq (1-t) \vol(K)^{1/n} + t \vol(L)^{1/n} ~,~ \forall t \in [0,1] ~,
\]
for all convex $K,L \subset \Real^n$. This is precisely the content of the celebrated Brunn--Minkowski inequality in Euclidean space (e.g. \cite{Schneider-Book,GardnerSurveyInBAMS}), at least for convex domains. Consequently, Theorem \ref{thm:Colesanti} provides yet another proof of the Brunn--Minkowski inequality in Euclidean space via the Reilly formula. Conceptually, this is not surprising since the generalized Reilly formula is in a sense a dual formulation of the Brascamp--Lieb inequality (see \cite{KolesnikovEMilmanReillyPart1}), and the latter is known to be an infinitesimal form of the Prekop\'a--Leindler inequality, which in turn is a functional (essentially equivalent) form of the Brunn-Minkowski inequality. So all of these inequalities are intimately intertwined and essentially equivalent to one another; see \cite{BrascampLiebPLandLambda1,BobkovLedoux, Ledoux-Book} for more on these interconnections.  We also mention that as a by-product, we may obtain all the well-known consequences of the Brunn--Minkowski inequality (see \cite{GardnerSurveyInBAMS}); for instance, by taking the first derivative in $t$ above, we may deduce the (anisotropic) isoperimetric inequality (for convex sets $K$).

Since our generalization of Colesanti's Theorem holds on any weighted-manifold satisfying the $\CD(0,N)$ condition, it is then natural to similarly try and obtain a Brunn--Minkowski or isoperimetric-type inequality in the latter setting. The main difficulties arising with such an attempt in this generality are the lack of homogeneity, the lack of a previously known generalization of Minkowski addition, and the fact that enlargements of convex sets are in general non-convex (consider geodesic balls on the sphere which are extended past the equator). At least some of these issues are addressed in what follows.

\subsection{Riemannian Brunn--Minkowski for Geodesic Extensions}
 
Let $K$ denote a compact subset of $(M^n,g)$ with $C^2$ smooth boundary ($n \geq 2$) which is bounded away from $\partial M$. Denote:
\[
\delta^0(K) := \mu(K) ~,~ \delta^1(K) := \mu_{\partial K}(\partial K) := \int_{\partial K} d\mu_{\partial K} ~,~ \delta^2(K) := \int_{\partial K} H_\mu d\mu_{\partial K} ~.
\]
It is well-known (see e.g. \cite{EMilmanGeometricApproachPartI} or Subsection \ref{subsec:Full-BM}) that $\delta^i$, $i=0,1,2$, are the $i$-th variations of $\mu(K_t)$, where $K_t$ is the $t$-neighborhood of $K$, i.e. $K_t := \set{x \in M ;  d(x,K) \leq t}$ with $d$ denoting the geodesic distance on $(M,g)$. Given $1/N \in (-\infty,1/n]$,  denote in analogy to the Euclidean case the ``generalized quermassintegrals" by:
\[
W_N(K) = \delta^0(K) ~,~ W_{N-1}(K) = \frac{1}{N} \delta^1(K) ~,~ W_{N-2}(K) = \frac{1}{N (N-1)} \delta^2(K) ~.
\]
Observe that when $\mu = \vol_M$ and $N=n$, these quermassintegrals coincide with the Lipschitz--Killing invariants in Weyl's celebrated tube formula, namely the coefficients of the polynomial $\mu(K_t) = \sum_{i=0}^n {n \choose i} W_{n-i}(K) t^i$ for $t \in [0,\eps_K]$ and small enough $\eps_K > 0$. In particular, when $(M,g)$ is Euclidean and $K$ is convex, these generalized quermassintegrals coincide with their classical counterparts, discovered by Steiner in the 19th century (see e.g. \cite{BernigCurvatureSurvey} for a very nice account). 

\medskip
As an immediate consequence of Corollary \ref{cor:HuangRuan1} we obtain:

\begin{cor}{\bf (Riemannian Brunn-Minkowski for Geodesic Extensions)}. \label{cor:Geodesic-BM}
Assume that $(K,g|_K,\mu|_K)$ satisfies the $\CD(0,N)$ condition ($1/N \in (-\infty,1/n]$) and that $K$ is strictly locally-convex ($\II_{\partial K} > 0$). Then the following statements hold:
\begin{enumerate}
\item (Generalized Minkowski's second inequality for geodesic extensions)
\begin{equation} \label{eq:Alexandrov}
W_{N-1}(K)^2 \geq W_{N}(K) W_{N-2}(K) ~,
\end{equation}
or in other words:
\[
\delta^1(K)^2 \geq \frac{N}{N-1} \delta^0(K) \delta^2(K) ~.
\]
\item (Infinitesimal Geodesic Brunn--Minkowski)
$(d/dt)^2 N \mu(K_t)^{1/N} |_{t=0} \leq 0$. 
\item (Global Geodesic Brunn--Minkowski)
The function $t \mapsto N \mu(K_t)^{1/N}$ is concave on any interval $[0,T]$ so that for all $t \in [0,T)$, $K_t$ is $C^2$ smooth, strictly locally-convex, bounded away from $\partial M$, and $(K_t,g|_{K_t},\mu|_{K_t})$ satisfies $\CD(0,N)$. 
\end{enumerate}
\end{cor}
\begin{proof}
The first assertion is precisely the content of Corollary \ref{cor:HuangRuan1}. The second follows since:
\[
(d/dt)^2 N \mu(K_t)^{1/N} |_{t=0} = \mu(K_t)^{1/N-2} \brac{\delta_2(K) \delta_0(K) - \frac{N-1}{N} \delta_1(K)^2} ~.
\]
The third is an integrated version of the second. 
\end{proof}

\begin{rem}
In the non-weighted Riemannian setting, the interpretation of Corollary \ref{cor:HuangRuan1} as a Riemannian version of Minkowski's second inequality was already noted by Reilly \cite{ReillyMeanCurvatureEstimate}. 
We also mention that in Euclidean space, a related Alexandrov--Fenchel inequality was shown to hold by Guan and Li \cite{GuanLiAlexandrovFenchelForNonConvexUsingFlows} for arbitrary mean-convex star-shaped domains. 
\end{rem}

\subsection{Generalized Minkowski Addition: The Parallel Normal Flow} \label{subsec:Minkowski-Sum} 

Let $F_0 : \Sigma^{n-1} \rightarrow M^n$ denote a smooth embedding of an oriented submanifold $\Sigma_0 := F_0(\Sigma)$ in $(M,g)$, where $\Sigma$ is a $n-1$ dimensional compact smooth  oriented manifold without boundary. The following geometric evolution equation for $F : \Sigma \times [0,T] \rightarrow M$ has been well-studied in the literature:
\begin{equation} \label{eq:flow0}
\frac{d}{dt} F(y,t) = \varphi(y,t) \nu_{\Sigma_t}(F(y,t)) ~,~ F(y,0) = F_0 ~,~ y \in \Sigma ~,~ t \in [0,T] ~.
\end{equation}
Here $\nu_{\Sigma_t}$ is the unit-normal (in accordance to the chosen orientation) to $\Sigma_t := F_t(\Sigma)$, $F_t := F(\cdot,t)$, and $\varphi : \Sigma \times [0,T] \rightarrow \Real_+$ denotes a function depending on the extrinsic geometry of $\Sigma_t \subset M$ at $F(y,t)$. Typical examples for $\varphi$ include the mean-curvature, the inverse mean-curvature, the Gauss curvature, and other symmetric polynomials in the principle curvatures (see \cite{HuiskenPoldenSurvey} and the references therein).  

Motivated by the DeTurck trick in the analysis of Ricci-flow (e.g. \cite{ToppingRicciFlowBook}), we propose to add another tangential component $\tau_t$ to (\ref{eq:flow0}). Let $\varphi : \Sigma \rightarrow \Real$ denote a $C^2$ function which is fixed throughout the flow. Assume that $\II_{\Sigma_t} > 0$ for $t \in [0,T]$ along the following flow:
\begin{eqnarray}
 \label{eq:Mink-flow} \frac{d}{dt} F(y,t)  &=& \omega_t(F(y,t)) ~,~ F(y,0) = F_0 ~,~ y \in \Sigma ~,~ t \in [0,T] ~, \\
\nonumber \omega_t &:=& \varphi_t \nu_{\Sigma_t} + \tau_t ~ \text{ on $\Sigma_t$} ~,~ \tau_t := \II_{\Sigma_t}^{-1} \nabla_{\Sigma_t} \varphi_t  ~,~ \varphi_t := \varphi \circ F_t^{-1} ~.
\end{eqnarray}
For many flows, the tangential component $\tau_t$ would be considered an inconsequential diffeomorphism term, which does not alter the set $\Sigma_t = F_t(\Sigma)$, only the individual trajectories $t \mapsto F(y,t)$ for a given $y \in \Sigma$. However, contrary to most flows where $\varphi(y,t)$ depends solely on the geometry of $\Sigma_t$ at $F_t(y)$, for our flow $\varphi$ plays a different role, and in particular its value along every trajectory is fixed throughout the evolution. Consequently, this tangential term creates a desirable geometric effect as we shall see below. 

Before proceeding, it is useful to note that (\ref{eq:Mink-flow}) is clearly parametrization invariant: if $\zeta: \Sigma' \rightarrow \Sigma$ is a diffeomorphism and $F$ satisfies (\ref{eq:Mink-flow}) on $\Sigma$, then $F'(z,t) := F(\zeta(z),t)$ also satisfies (\ref{eq:Mink-flow}) with $\varphi'(z) := \varphi(\zeta(z))$. Consequently, we see that (\ref{eq:Mink-flow}) defines a semi-group of pairs $(\Sigma_t, \varphi_t)$, so it is enough to perform calculations at time $t=0$. In addition, we are allowed to use a convenient parametrization $\Sigma'$ for our analysis. 

\subsubsection{Euclidean Setting}

We now claim that in Euclidean space, Minkowski summation can indeed be parametrized by the evolution equation (\ref{eq:Mink-flow}). Given a convex compact set in $\Real^n$ containing the origin in its interior (``convex body") with $C^2$ smooth boundary and outer unit-normal $\nu_K$, by identifying $T_x \Real^n$ with $\Real^n$, $\nu_K: \partial K \rightarrow S^{n-1}$ is the Gauss-map. Note that when $K$ is strictly convex ($\II_{\partial K} > 0$), the Gauss-map is a diffeomorphism. Finally, the support function $h_K$ is defined by $h_K(x) := \sup_{y \in K} \scalar{x,y}$, so that $\scalar{x,\nu_K(x)} = h_K(\nu_K(x))$ and $\scalar{\nu_K^{-1}(\nu),\nu} = h_K(\nu)$. 

\begin{prop} \label{prop:Euclidean-Flow}
Let $K$ and $L$ denote two strictly convex bodies in $\Real^n$ with $C^2$ smooth boundaries. Let $F : S^{n-1} \times \Real_+ \rightarrow \Real^n$ be defined by $F(\nu,t) := \nu_{K + tL}^{-1}(\nu) $, so that $\partial (K + t L) =  F_t(S^{n-1})$ for all $t \geq 0$.  Then $F$ satisfies (\ref{eq:Mink-flow}) with $\varphi = h_L$ and $F_0 := \nu_K^{-1}$. 
\end{prop}

\begin{proof}
As the support function is additive with respect to Minkowski addition, then so is the inverse Gauss-map: $\nu_{K + tL}^{-1} = \nu_K^{-1} + t \nu_L^{-1}$. Consequently:
\[
\frac{d}{dt} F_t(\nu) = \nu_L^{-1}(\nu) = h_L(\nu) \nu + \brac{ \nu_L^{-1}(\nu) - \scalar{\nu_L^{-1}(\nu) , \nu} \nu} ~.
\]
Since $\nu_{F_t(S^{n-1})}(F_t(\nu)) = \nu$, it remains to show (in fact, just for $t=0$) with the usual identification between $T_x \Real^n$ and $\Real^n$ that:
\[
 \nu_{K + t L}(x) = \nu \;\;\; \Rightarrow \;\;\; \II_{\partial(K + t L)}^{-1} \nabla_{\partial(K + tL)} (h_L \circ \nu_{K+tL}) (x) = \nu_L^{-1}(\nu) - \scalar{\nu_L^{-1}(\nu) , \nu} \nu ~.
 \]
 Indeed by the chain-rule:
 \[
 \nabla_{\partial K}(h_L(\nu_K(x))) = \nabla_{S^{n-1}} h_L(\nu) \nabla_{\partial K} \nu_K(x) = \II_{\partial K}(x)  \nabla_{S^{n-1}} h_L(\nu) ~,
 \]
 so our task reduces to showing that:
 \[
 \nabla_{S^{n-1}} h_L(\nu) = \nu_L^{-1}(\nu) - \scalar{\nu_L^{-1}(\nu) , \nu} \nu ~,~ \forall \nu \in S^{n-1} ~.
 \]
 This is indeed the case, and moreover:
\[
\nabla_{\Real^n} h_L(\nu) = \nu_L^{-1}(\nu) ~,~ \forall \nu \in S^{n-1} ~.
\]
The reason is that $\nu_L(x) = \frac{\nabla_{\Real^n} \norm{x}_L}{\abs{\nabla_{\Real^n}\norm{x}_L}}$, and since $\nabla_{\Real^n} h_L$ is $0$-homogeneous, we obtain:
\[
\nabla_{\Real^n} h_L \circ \nu_L(x) = \nabla_{\Real^n} h_L \circ \nabla_{\Real^n} \norm{x}_L = x ~,~ \forall x \in \partial L  ~,
\]
where the last equality follows since $h_L$ and $\norm{\cdot}_L$ are dual norms. This concludes the proof. 
\end{proof}

\subsubsection{Characterization in Riemannian Setting}

The latter observation gives a clear geometric interpretation of what the flow (\ref{eq:Mink-flow}) is doing in the Euclidean setting: normals to the evolving hypersurface remain constant along  trajectories. In the more general Riemannian setting, where one cannot identify between $M$ and $T_x M$ and where the Gauss map is not available, we have the following geometric characterization of the flow (\ref{eq:Mink-flow}) which extends the latter property: normals to the evolving hypersurface remain \emph{parallel} along  trajectories. Consequently, we dub (\ref{eq:Mink-flow}) the ``Parallel Normal Flow".

\begin{prop} \label{prop:normal-parallel}
Consider the following geometric evolution equation along a time-dependent vector-field $\omega_t$:
\[
\frac{d}{dt} F_t(y) = \omega_t(F_t(y))  ~,~ y \in \Sigma ~,~ t \in [0,T] ~,
\]
and assume that $F_t: \Sigma \rightarrow \Sigma_t$ is a local diffeomorphism and that $\II_{\Sigma_t} > 0$ for all $t \in [0,T]$. 
Then the unit-normal field is parallel along the flow:
\[
\frac{d}{dt} \nu_{\Sigma_t}(F_t(y)) = 0 ~,~ \forall y \in \Sigma ~,~ \forall t \in [0,T] ~,
\]
if and only if there exists a family of functions $f_t : \Sigma_t\rightarrow \Real$, $t \in [0,T]$, so that:
\begin{equation} \label{eq:char}
\omega_t = f_t \nu_{\Sigma_t} + \II_{\Sigma_t}^{-1} \nabla_{\Sigma_t} f_t ~,
\end{equation}
Furthermore, the entire normal component of $\omega_t$, denoted $\omega^\nu_t := \scalar{\omega_t,\nu_{\Sigma_t}} \nu_{\Sigma_t}$, is parallel along the flow:
\[
\frac{d}{dt} \omega^\nu_{t}(F_t(y)) = 0 ~,~ \forall y \in \Sigma ~,~ \forall t \in [0,T] ~,
\]
if and only if there exists a function $\varphi : \Sigma \rightarrow \Real$ so that $f_t = \varphi \circ F_t^{-1}$ in (\ref{eq:char}). 
\end{prop}

Recall that the derivative of a vector-field $X$ along a path $t \mapsto \gamma(t)$ is interpreted by employing the connection $\frac{d}{dt} X(\gamma(t)) = \nabla_{\gamma'(t)} X$. 

\begin{proof}
First, observe that $\scalar{\frac{d}{dt} \nu_{\Sigma_t}(F_t(y))  , \nu_{\Sigma_t}} = \frac{1}{2} \frac{d}{dt} \scalar{\nu_{\Sigma_t},\nu_{\Sigma_t}}(F_t(y)) = 0$, so $\frac{d}{dt} \nu_{\Sigma_t}(F_t(y))$ is tangent to $\Sigma_t$. 
Given $y \in \Sigma$ let $e \in T_y\Sigma$ and set $e_t := dF_t(e) \in T_{F_t(y)} \Sigma_t$. Since $\scalar{\nu_{\Sigma_t} , e_t} = 0$, we have:
\[
\scalar{\frac{d}{dt} \nu_{\Sigma_t}(F_t(y)) , e_t} = - \scalar{\nu_{\Sigma_t}, \frac{d}{dt} dF_t(e)} =  -\scalar{\nu_{\Sigma_t} , \nabla_{e_t} \frac{d}{dt} F_t(y)} = -\scalar{\nu_{\Sigma_t} , \nabla_{e_t} \omega_t} ~.
\]
Decomposing $\omega_t$ into its normal $\omega^\nu_t = f_t \nu_{\Sigma_t}$ and tangential $\omega^\tau_t$ components, we calculate:
\[
-\scalar{\nu_{\Sigma_t} , \nabla_{e_t} \omega_t} = -  \nabla_{e_t} f_t - f_t \scalar{\nu_{\Sigma_t}, \nabla_{e_t} \nu_{\Sigma_t}} - \scalar{\nu_{\Sigma_t} , \nabla_{e_t} \omega^\tau_t} ~.
\]
Since $\scalar{\nu_{\Sigma_t}, \nabla_{e_t} \nu_{\Sigma_t}} = \frac{1}{2} \nabla_{e_t} \scalar{\nu_{\Sigma_t}, \nu_{\Sigma_t}} = 0$, $\scalar{\nu_{\Sigma_t} , \nabla_{e_t} \omega^\tau_t} = -\scalar{\II_{\Sigma_t} \omega^\tau_t,e_t}$ and since $e$ and hence $e_t$ were arbitrary, we conclude that:
\[
\frac{d}{dt} \nu_{\Sigma_t}(F_t(y))  = -\nabla_{\Sigma_t} f_t + \II_{\Sigma_t} \omega^\tau_t ~,
\]
and so the first assertion follows. The second assertion follows by calculating:
\[
\frac{d}{dt} \omega^\nu_{t}(F_t(y)) = \frac{d}{dt} (f_t \nu_{\Sigma_t})(F_t(y)) = (\frac{d}{dt} f_t(F_t(y))) \nu_{\Sigma_t}(F_t(y)) + f_t(F_t(y)) \frac{d}{dt} \nu_{\Sigma_t}(F_t(y)) ~.
\]
We see that $\omega^\nu_{t}$ is parallel along the flow if and only if both normal and tangential components on the right-hand-side above vanish, reducing to the first assertion in conjunction with the requirement that $f_t$ remain constant along the flow, i.e. $f_t = \varphi \circ F_t^{-1}$. 
\end{proof}

Consequently, given a strictly locally-convex compact set $\Omega \subset (M,g)$ with $C^2$ smooth boundary which is bounded away from $\partial M$, the region bounded by $F_t(\partial \Omega) \subset (M,g)$ with initial conditions $F_0 = Id$ and $\varphi \in C^2(\partial \Omega)$, if it exists, will be referred to as the ``Riemannian Minkowski Extension of $\Omega$ by $t \varphi$" and denoted by $\Omega + t \varphi$. 
Note that this makes sense as long as the Parallel Normal Flow is a diffeomorphism which preserves the aforementioned convexity and boundedness away from $\partial M$ up until time $t$ - we will say in that case that the Riemannian Minkowski extension is ``well-posed". When $\varphi \equiv 1$ on $\partial \Omega$, we obtain the usual geodesic extension $\Omega_t$. Note that $\varphi$ need not be positive to make sense of this operation, and that multiplying $\varphi$ by a positive constant is just a time re-parametrization of the flow. Also note that this operation only depends on the geometry of $(M,g)$, and not on the measure $\mu$, in accordance with the classical Euclidean setting. 

\subsubsection{Homogeneous Monge-Amp\`ere Equation and Short-Time Existence in Analytic Case}

We now briefly explain the relation of our flow to a homogeneous Monge-Amp\`ere equation on the exterior of $\Omega$. 
Consider the function $u$ whose $t$-level sets are precisely the hypersurfaces $\Sigma_t \subset M$:
\[
u(F_t(y))=t \;\;, \;\;   y\in \Sigma .  
\]
Differentiating this formula in $t$ yields  $1 = \scalar{ \nabla u(F_t(y)), \frac{d}{dt} F_t(y) } = \scalar{\nabla u(F_t(y)), \omega_t(F_t(y))}$. 
But since $\nu_{\Sigma_t}(F_t(y)) = \frac{\nabla u(F_t)}{|\nabla u(F_t)|}$, one immediately obtains:
\begin{equation} \label{eq:u-parallel}
|\nabla u(F_t(y))| = \frac{1}{\varphi(y)}.
\end{equation}

\begin{corollary}
The vector $\nabla u(F_t(y))$ is parallel along the flow.
\end{corollary}
\begin{proof}
Immediate since $\nabla u(F_t(u)) = \abs{\nabla u(F_t(y))} \nu_{\Sigma_t}(F_t(y))$. $\abs{\nabla u(F_t(y))}$ is constant along the flow by (\ref{eq:u-parallel}), and $\nu_{\Sigma_t}(F_t(y))$ is parallel along the flow by Proposition \ref{prop:normal-parallel}. 
\end{proof}

In particular, we deduce that $\omega_t$ is a zero eigenvector for $\nabla^2 u$, as $\nabla^2 u \cdot \omega_t =
\nabla_{\omega_t} \nabla u=0$.  We therefore see that $u$ solves the following homogeneous Monge-Amp\`ere boundary value problem:
\begin{equation} \label{eq:MA}
\det \nabla^2 u \equiv 0, \ u|_{\Sigma_0} \equiv 0, \ u_\nu |_{\Sigma_0}(x) = \frac{1}{\varphi_0(x)},
\end{equation}
and that the trajectories of the Normal Parallel Flow are precisely the characteristic curves of this PDE. 

The short-time solution for the above boundary value problem (and hence our original flow) can be established by a standard application of the Cauchy--Kowalevskaya theorem, provided that all the data (including the metric) is analytic and $\varphi_0>0$, $\II_{\Sigma_0}>0$; further investigation into the short-time existence will be conducted elsewhere (cf. \cite{RubinsteinZelditch-CauchyProblemForHMA-II,RubinsteinZelditch-CauchyProblemForHMA-III}). We conclude this subsection with a couple of additional observations regarding the classical Euclidean setting. 

\smallskip

The relation between the homogeneous Monge-Amp\`ere equation and interpolation or infimum-convolution in the context of Banach Space Theory or K\"{a}hler geometry is known (e.g. \cite{Semmes-InterpolationOfBanachSpaces,CorderoKlartag-Interpolation,RubinsteinZelditch-CauchyProblemForHMA-II,RubinsteinZelditch-CauchyProblemForHMA-III,ArtsteinRubinstein-Polarity}). However,  even in the Euclidean setting, the above explicit description seems to have been previously unnoted in the literature: given two compact convex bodies $K,L \subset \Real^n$ containing the origin (say smooth and strictly convex), the entire family $\set{K + tL}_{t \geq 0}$ is obtained as the $t$-sub-level sets of a (convex) solution $u$ to the homogeneous Monge-Amp\`ere equation (\ref{eq:MA}) on $\Real^n \setminus K$ with $\Sigma_0 = \partial K$ and $\varphi_0 = h_L \circ \nu_{\partial K}$. 

 Similarly, if $K \subset L$ are convex as above, it follows that the homogeneous Monge-Amp\`ere equation:
\[
\det \nabla^2 u \equiv 0, \ u|_{\partial K} \equiv 0, \  u|_{\partial L} \equiv 1 ,
\]
admits a convex solution $u : L \setminus K \rightarrow [0,1]$ whose $t$-sub-level sets ($t \in [0,1]$) are precisely $(1-t) K + t L$. Indeed, the corresponding Normal Parallel Flow is then $\partial K \ni x \mapsto (1-t) x + t \nu^{-1}_{\partial L} \circ \nu_{\partial K}(x)$ mapping $\partial K$ onto $\partial ((1-t) K + t L)$.
The analogous statement in the Riemannian setting will be developed elsewhere. 

\smallskip

Lastly, we obtain using the Parallel Normal Flow an explicit map $T : K \rightarrow L$ so that $(Id + t T)(K) = K + t L$ for all $t \geq 0$; simply set:
\[
 T(x) =  \norm{x}_K \nu_{\partial L}^{-1} \circ \nu_{\partial K}( x / \norm{x}_K) ,
\]
and use homogeneity and the fact that $(Id + t T)(\partial K) = \partial ( K + t L)$. A non-explicit map using Optimal-Transport has been previously constructed by Alesker, Dar and V. Milman \cite{AleskerDarMilman-RemarkableDiffeo}, and it would be interesting to see how to use our explicit map for deducing Alexandrov--Fenchel inequalities, as in \cite{AleskerDarMilman-RemarkableDiffeo}.

\subsection{Riemannian Brunn-Minkowski} \label{subsec:Full-BM}

We have seen that Riemannian Minkowski extension coincides with Minkowski summation in the Euclidean setting: $K + t \cdot h_L = K + t L$. We do not go here into analyzing the well-posedness of this operation in the general (non-analytic) case, but rather concentrate on using this operation to derive the following Riemannian generalization of the Brunn-Minkowski  inequality. 

\begin{thm}[Riemannian Brunn--Minkowski Inequality] \label{thm:Full-BM}
Let $\Omega \subset (M,g)$ denote a strictly locally-convex ($\II_{\partial \Omega} > 0$) compact set with $C^2$ smooth boundary which is bounded away from $\partial M$, and let $\varphi \in C^2(\partial \Omega)$. Let $\Omega_t$ denote the Riemannian Minkowski extension $\Omega_t := \Omega + t \varphi$, and assume that it is well-posed for all $t \in [0,T]$. Assume that $(M,g,\mu)$ satisfies the $\CD(0,N)$ condition ($1/N \in (-\infty,1/n]$). Then the function
\[
t \mapsto N \mu(\Omega_t)^{1/N} 
\]
is concave on $[0,T]$. 
\end{thm}
\begin{proof}
Set $\Sigma := \partial \Omega$, $F_0 = Id$, and recall that our evolution equation is:
\begin{equation} \label{eq:Jac-Flow}
\frac{d}{dt} F_t(y) = \omega_t(F_t(y))  := \varphi(y) \nu_{\Sigma_t}(F_t(y)) + \tau_t(F_t(y)) ~,~ F_0(y) = Id ~,~ y \in \Sigma ~,~ t \in [0,T] ~.
\end{equation}
As previously explained, it is enough to perform all the analysis at time $t=0$. 

Clearly, the first variation of $\mu(\Omega_t)$ only depends on the normal velocity $\varphi$, and so we have:
\[
\frac{d}{dt} \mu(\Omega_t)|_{t=0} = \int_{\Sigma} \varphi \exp(-V) d\vol_{\Sigma} ~.
\]
By the semi-group property, it follows that:
\[
\frac{d}{dt} \mu(\Omega_t) = \int_{\Sigma_t} \varphi \circ F_t^{-1} \exp(-V) d\vol_{\Sigma_t} ~.
\]
 Since $F_t$ is a diffeomorphism for small $t \geq 0$, we obtain by the change-of-variables formula:
\begin{equation} \label{eq:Jac-base}
\frac{d}{dt} \mu(\Omega_t) = \int_{\Sigma} \varphi \exp(-V \circ F_t) \; \Jac F_t \; d\vol_{\Sigma} ~,
\end{equation}
where $\Jac F_t(y)$ denotes the Jacobian of $(\Sigma,g|_{\Sigma}) \ni y \mapsto F_t(y) \in (\Sigma_t, g|_{\Sigma_t})$, i.e. the determinant of $d_y F_t : (T_y \Sigma,g_y) \rightarrow (T_{F_t(y)} \Sigma_t, g_{F_t(y)})$. 

As is well known, $\frac{d}{dt} \Jac F_t = \div_{\Sigma_t} \frac{d}{dt} F_t$; we briefly sketch the argument. It is enough to show this for $t=0$ and for a $y \in \Sigma$ so that $\scalar{\frac{d}{dt} F_t(y) , \nu_{\Sigma}(y)} \neq 0$. Fix an orthonormal frame $e_1,\ldots,e_n$ in $T M$ so that $e_n$ coincides with $\nu_{\Sigma_t}$ in a neighborhood of $(y,0)$ in $M \times \Real$, and hence $e_1,\ldots,e_{n-1}$ is a basis for $T_{F(y,t)} \Sigma$. Since $d F_0 = Id$, it follows that:
\[
\frac{d}{dt} \Jac F_t(y)  = tr \brac{\frac{d}{dt} d_y F_t } = \sum_{i=1}^{n-1} \frac{d}{dt} \scalar{d_y F_t(e_i(y)) , e_i(F_t(y))} ~.
\]
Now as $F_0 = Id$ and $\frac{d}{dt} F_t|_{t=0} = \omega_0$, we have at $(y,0)$ (denoting $\omega = \omega_0$):
\[
\frac{d}{dt} \scalar{d_y F_t(e_i) , e_i(F_t(y))}|_{t=0} = \scalar{\frac{d}{dt} d_y F_t(e_i) |_{t=0} , e_i} + \scalar{e_i , \nabla_{\frac{d}{dt} F_t|_{t=0}} e_i}  = \scalar{\nabla_{e_i} \omega, e_i} + \scalar{e_i , \nabla_{\omega} e_i} ~.
\]
But $\scalar{e_i, \nabla_{\omega} e_i} = \frac{1}{2} \nabla_{\omega} \scalar{e_i,e_i} = 0$, and so we confirm that $\frac{d}{dt} \Jac F_t = \div_{\Sigma_t} \omega_t$. 

Now, taking the derivative of (\ref{eq:Jac-base}) in $t$, we obtain:
\[
\frac{d^2}{(dt)^2} \mu(\Omega_t)|_{t=0} = \int_{\Sigma} \varphi  (\div_{\Sigma} \omega - \scalar{\nabla V,\omega}) \exp(-V) d\vol_{\Sigma} ~.
\]
Recall that $\omega = \varphi \nu_{\Sigma} + \tau$ (denoting $\tau = \tau_0$), so:
\[
 \div_{\Sigma} \omega  - \scalar{\nabla V,\omega} = \varphi (H_{\Sigma,g} - \scalar{\nabla V,\nu_\Sigma}) + \div_{\Sigma} \tau - \scalar{\nabla_{\Sigma} V, \tau}  =  \varphi H_{\Sigma,\mu} + \div_{\Sigma,\mu} \tau ~.
\]
 Plugging this above and integrating by parts, we obtain:
 \[
 \frac{d^2}{(dt)^2} \mu(\Omega_t)|_{t=0} = \int_{\Sigma} H_{\Sigma,\mu} \varphi^2 d\mu_\Sigma - \int_{\Sigma} \scalar{\nabla_{\Sigma} \varphi,\tau} d\mu_\Sigma .
 \]
Recalling that $\tau = \II_{\Sigma}^{-1} \nabla_{\Sigma} \varphi$ and
applying Theorem \ref{thm:Colesanti}, we deduce that:
\[
\frac{d^2}{(dt)^2} \mu(\Omega_t) |_{t=0} \leq \frac{N-1}{N} \frac{ (\int_{\Sigma} \varphi d\mu_\Sigma)^2 }{\mu(\Omega)} = \frac{N-1}{N} \frac{\brac{\frac{d}{dt} \mu(\Omega_t)|_{t=0}}^2}{\mu(\Omega)} ~,
\]
which is precisely the content of the assertion.
\end{proof}

\begin{rem} \label{rem:other-gen-BM}
Other more standard generalizations of the Brunn--Minkowski inequality in the weighted Riemannian setting and in the even more general metric-measure space setting, for spaces satisfying the $\CD(\rho,N)$ condition, have been obtained by Cordero-Erausquin--McCann--Schmuckenshl\"{a}ger \cite{CMSInventiones,CMSManifoldWithDensity}, Sturm \cite{SturmCD12}, Lott--Villani \cite{LottVillaniGeneralizedRicci} and Ohta \cite{Ohta-NegativeN}, using the theory of optimal-transport. In those versions, Minkowski interpolation $(1-t) K + t L$ is replaced by geodesic interpolation of two domains, an operation whose existence does not require any a-priori justification, and which is not confined to convex domains. However, our version has the advantage of extending Minkowski summation $K + t L$ as opposed to interpolation, so we just need a single domain $\Omega_0$ and an initial condition $\varphi_0$ on the normal derivative to $\partial \Omega_0$; this may consequently be better suited for compensating the lack of homogeneity in the Riemannian setting and obtaining isoperimetric inequalities. There seem to be some interesting connections between the Parallel Normal Flow and an appropriate optimal-transport problem and Monge--Amp\`ere equation, but this is a topic for a separate note. In this context, we mention the work by V. I. Bogachev and the first named author \cite{BogachevKolesnikovFlow}, who showed a connection between the Gauss curvature flow and an appropriate optimal transport problem. 
\end{rem}

\begin{rem}
While we do not go into this here, it is clear that in analogy to the Euclidean setting, one may use the Riemannian Minkowski extension operation to define the $k$-th Riemannian mixed volume of a strictly locally-convex $K$ and $\varphi \in C^2(\partial K)$ by taking the $k$-th variation of $t \mapsto \mu(K + t \varphi)$ and normalizing appropriately. It is then very plausible to expect that these mixed volumes should satisfy Alexandrov--Fenchel type inequalities, in analogy to the original inequalities in the Euclidean setting. 
\end{rem}

\subsection{Comparison with the Borell--Brascamp--Lieb Theorem} \label{subsec:BBL}

Let $\mu$ denote a Borel measure with convex support $\Omega$ in Euclidean space $(\Real^n,\abs{\cdot})$. In this Euclidean setting, it was shown by Borell \cite{BorellConvexMeasures} and independently by Brascamp and Lieb \cite{BrascampLiebPLandLambda1}, that if $(\Omega,\abs{\cdot},\mu)$ satisfies $\CD(0,N)$, $1/N \in [-\infty,1/n]$, then for all Borel subsets $A,B \subset \Real^n$ with $\mu(A),\mu(B) > 0$:
\begin{equation} \label{eq:BBL}
\mu((1-t) A + t B) \geq \brac{(1-t) \mu(A)^{\frac{1}{N}} + t \mu(B)^{\frac{1}{N}}}^N ~,~ \forall t \in [0,1] ~.
\end{equation}
Consequently, since $(1-t) K + t K = K$ when $K$ is convex, by using $A = K + t_1 L$ and $B = K + t_2 L$ for two convex subsets $K,L\subset \Omega$, it follows that the function:
\[
t \mapsto N \mu(K + t L)^{\frac{1}{N}} 
\]
is concave on $\Real_+$. Clearly, Corollary \ref{cor:Geodesic-BM} and Theorem \ref{thm:Full-BM} are generalizations to the Riemannian setting of this fact, and in particular provide an alternative proof in the Euclidean setting. The above reasoning perhaps provides some insight as to the reason behind the restriction to convex domains in the concavity results of this section. 

We mention in passing that when the measure $\mu$ is homogeneous (in the Euclidean setting), one does not need to restrict to convex domains, simply by rescaling $A$ in (\ref{eq:BBL}). See \cite{EMilmanRotemHomogeneous} for isoperimetric applications. 

\subsection{The Weingarten Curvature Wave Equation} \label{subsec:strange-flow}

To conclude this section, we observe that there is another natural evolution equation which yields the concavity of $N \mu(\Omega_t)^{1/N}$. 
Assume that $\varphi$ in (\ref{eq:flow0}) evolves according to the following heat-equation on the evolving weighted-manifold $(\Sigma_t,\II_{\Sigma_t},\mu_{\Sigma_t})$ equipped the Weingarten metric $\II_{\Sigma_t}$ and the measure $\mu_{\Sigma_t} := \exp(-V) d\vol_{g_{\Sigma_t}}$, $g_{\Sigma_t} := g|_{\Sigma_t}$:
\begin{equation} \label{eq:flow1}
\frac{d}{dt} \log \varphi(y,t) = L_{(\Sigma_t,\text{\rm{II}}_{\Sigma_t}, \mu_{\Sigma_t})}(\varphi_t)(F_t(y)) ~,~ \varphi_t := \varphi(F_t^{-1}(\cdot),t) ~,~ \varphi(\cdot,0) = \varphi_0 ~.
\end{equation}
Here $L = L_{(\Sigma_t,\text{\rm{II}}_{\Sigma_t}, \mu_{\Sigma_t})}$ denotes the weighted-Laplacian operator associated to this weighted-manifold, namely:
\begin{equation} \label{eq:flow-L}
L(\psi) = \text{div}_{\text{\rm{II}}_{\Sigma_t},\mu}(\nabla_{\rm{II}_{\Sigma_t}} \psi) = \text{div}_{g_{\Sigma_t},\mu}(\rm{II}^{-1}_{\Sigma_t} \nabla_{g_{\Sigma_t}} \psi) ~.
\end{equation}
The last transition in (\ref{eq:flow-L}) is justified since for any test function $f$:
\begin{multline*}
\int_{\Sigma_t} f \cdot \text{div}_{\rm{II}_{\Sigma_t},\mu}(\nabla_{\rm{II}_{\Sigma_t}} \psi) d\mu_{\Sigma_t} =  - \int_{\Sigma_t} \rm{II}_{\Sigma_t}(\nabla_{\textrm{II}_{\Sigma_t}} \textit{f},\nabla_{\rm{II}_{\Sigma_t}} \psi) d\mu_{\Sigma_t} \\
= - \int_{\Sigma_t} g_{\Sigma_t}(\nabla_{g_{\Sigma_t}} f, \rm{II}^{-1}_{\Sigma_t}  \nabla_{g_{\Sigma_t}} \psi) d\mu_{\Sigma_t} = \int_{\Sigma_t} \textit{f} \cdot \text{div}_{g_{\Sigma_t},\mu}(\II_{\Sigma_t}^{-1} \nabla_{g_{\Sigma_t}} \psi) d\mu_{\Sigma_t} ~.
\end{multline*}
Note that (\ref{eq:flow1}) is precisely the (logarithmic) gradient flow in $L^2(\Sigma_t,\mu_{\Sigma_t})$ for the Dirichlet energy functional on $(\Sigma_t,\rm{II}_{\Sigma_t}, \mu_{\Sigma_t})$:
\[
\varphi \mapsto E(t,\varphi) := \frac{1}{2} \int_{\Sigma_t} \rm{II}_{\Sigma_t}(\nabla_{\rm{II}_{\Sigma_t}} (\varphi_t) , \nabla_{\rm{II}_{\Sigma_t}} (\varphi_t)) d\mu_{\Sigma_t} ~.
\]

Coupling (\ref{eq:flow0}) and (\ref{eq:flow1}), it seems that an appropriate name for the resulting flow would be the ``Weingarten Curvature Wave Equation", since the second derivative in time of $F_t$ in the normal direction to the evolving hypersurface $\Sigma_t$ is equal to the weighted Laplacian on $(\Sigma_t,\II_{\Sigma_t},\mu_{\Sigma_t})$. 
We do not go at all into justifications of existence of such a flow, but rather observe the following:

\begin{thm}[Weingarten Curvature Wave Equation is $N$-concave]
Assume that there exists a smooth solution $(F,\varphi)$ to the system of coupled equations (\ref{eq:flow0}) and (\ref{eq:flow1}) on $\Sigma \times [0,T]$, so that $F_t : \Sigma \rightarrow \Sigma_t \subset (M,g)$ is a diffeomorphism for all $t \in [0,T]$. Assume that $(M,g,\mu)$ satisfies the $\CD(0,N)$ condition ($1/N \in (-\infty,1/n]$), that $\Sigma_t$ are strictly-convex ($\II_{\Sigma_t} > 0$) and bounded away from $\partial M$ for all $t \in [0,T]$. Assume that $\Sigma_t$ are the boundaries of compact domains $\Omega_t$ having $\nu_t$ as their exterior unit-normal field.  Then the function
\[
t \mapsto  N \mu(\Omega_t)^{1/N} 
\]
is concave on $[0,T]$. 
\end{thm}
\begin{proof}
Denote $\Phi_t := \frac{d}{dt} \varphi(F_t^{-1}(\cdot),t)$. It is easy to verify as in the proof of Theorem \ref{thm:Full-BM} that:
\[
\frac{d}{dt} \mu(\Omega_t) = \int_{\partial \Omega_t} \varphi_t d\mu_{\partial \Omega_t} ~,~ \frac{d}{dt} \mu(\partial \Omega_t) = \int_{\partial \Omega_t} H_\mu \varphi_t d\mu_{\partial \Omega_t} ~,
\]
and that:
\[
\frac{d^2}{(dt)^2} \mu(\Omega_t) = \int_{\partial \Omega_t} (\Phi_t  + H_\mu \varphi_t^2) d\mu_{\partial \Omega_t} ~.
\]
Plugging the evolution equation (\ref{eq:flow1}) above and integrating by parts, we obtain:
\begin{eqnarray*}
\frac{d^2}{(dt)^2} \mu(\Omega_t) &=& \int_{\partial \Omega_t} ( \varphi_t \text{div}_{g_{\Sigma_t},\mu} \brac{\rm{II}^{-1}_{\Sigma_t}  \nabla_{g_{\Sigma_t}} \varphi_t} + H_\mu \varphi_t^2) d\mu_{\partial \Omega_t} \\
& = & \int_{\partial \Omega_t} (H_\mu \varphi_t^2 - \scalar{\II^{-1}_{\Sigma_t} \nabla_{g_{\Sigma_t}} \varphi_t, \nabla_{g_{\Sigma_t}} \varphi_t}) d\mu_{\partial \Omega_t} ~.
\end{eqnarray*}
Applying Theorem \ref{thm:Colesanti}, we deduce that:
\[
\frac{d^2}{(dt)^2} \mu(\Omega_t) \leq \frac{N-1}{N} \frac{ (\int_{\partial \Omega_t} \varphi_t d\mu_{\partial \Omega_t})^2 }{\mu(\Omega_t)} = \frac{N-1}{N} \frac{\brac{\frac{d}{dt} \mu(\Omega_t)}^2}{\mu(\Omega_t)} ~,
\]
which is precisely the content of the assertion.
\end{proof}

\section{Isoperimetric Applications} \label{sec:Apps}

We have seen in the previous section that under the $\CD(0,N)$ condition and for various geometric evolution equations, including geodesic extension, the function $t \mapsto N \mu(\Omega_t)^{1/N}$ is concave as long as $\Omega_t$ remain strictly locally-convex, $C^2$ smooth, and bounded away from $\partial M$. Consequently, the following derivative exists in the wide-sense:
\[
\mu^+(\Omega) := \frac{d}{dt} \mu(\Omega_t)|_{t=0} = \lim_{t \rightarrow 0+}\frac{\mu(\Omega_{t}) - \mu(\Omega)}{t} ~.
\]
$\mu^+(\Omega)$ is the induced ``boundary measure" of $\Omega$ with respect to $\mu$ and the underlying evolution $t \mapsto \Omega_t$. It is well-known and easy to verify (as in the proof of Theorem \ref{thm:Full-BM}) that in the case of geodesic extension, $\mu^+(\Omega)$ coincides with $\mu_{\partial \Omega}(\partial \Omega)$. We now mention several useful isoperimetric consequences of the latter concavity. For simplicity, we illustrate this in the Euclidean setting, but note that all of the results remain valid in the Riemannian setting as long as the corresponding generalizations described in the previous section are well-posed. 

Denote by $\mu^+_L(K)$ the boundary measure of $K$ with respect to $\mu$ and the Minkowski extension $t \mapsto K + t L$, where $L$ is a compact convex set having the origin in its interior.

\begin{prop} \label{prop:isop-app}
Let Euclidean space $(\Real^n,\abs{\cdot})$ be endowed with a measure $\mu$ with convex support $\Omega$, so that $(\Omega,\abs{\cdot},\mu)$ satisfies the $\CD(0,N)$ condition ($1/N \in (-\infty,1/n]$). Let $K \subset \Omega$ and $L \subset \Real^n$ denote two strictly convex compact sets with non-empty interior and $C^2$ boundary. Then:
\begin{enumerate}
\item
The function $t \mapsto N \mu(K + t L)^{1/N}$ is concave on $\Real_+$. 
\item
The following isoperimetric inequality holds:
\[
\mu^+_L(K) \geq \mu(K)^{\frac{N-1}{N}} \sup_{t > 0} N \frac{\mu(K + t L)^{1/N} - \mu(K)^{1/N}}{t} ~.
\]
In particular, if the $L$-diameter of $\Omega$ is bounded above by $D < \infty$ ($\Omega - \Omega \subset D L$), we have:
\[
\mu^+_L(K) \geq \frac{N}{D} \mu(K)^{\frac{N-1}{N}} \brac{\mu(\Omega)^{1/N} - \mu(K)^{1/N}} ~.
\]
Alternatively,  if $\mu(\Omega) = \infty$ and $N \in [n,\infty]$, we have: 
\begin{equation} \label{eq:for-hom}
\mu^+_L(K) \geq \mu(K)^{\frac{N-1}{N}}  \limsup_{t \rightarrow \infty} \frac{N \mu(t L)^{1/N}}{t} ~.
\end{equation}
\item
Define the following ``convex isoperimetric profile":
\[
\I^c_L(v) := \inf \set{\mu^+_L( K)  \; ;\;  \mu(K) = v \text{ , $K \subset \Omega$ has $C^2$ smooth boundary and $\II_{\partial K} > 0$ } } ~.
\]
Then the function $v \mapsto (\I^c_L(v))^{\frac{N}{N-1}}/v$ is non-increasing on its domain. 
\end{enumerate}
\end{prop}

\begin{rem}
Given a weighted-manifold $(M,g,\mu)$, recall that the usual isoperimetric profile is defined as:
\[
\I(v) := \inf \set{\mu(\partial A) \;;\;  \mu(A) = v \text{ , $A \subset M$ has $C^2$ smooth boundary}} ~.
\]
When $(M,g,\mu)$ satisfies the $\CD(0,N)$ condition with $N \in [n,\infty]$ and $\II_{\partial M} \geq 0$ ($M$ is locally convex), it is known that $v \mapsto \I(v)^{\frac{N}{N-1}}$ is in fact concave on its domain, implying that $v \mapsto \I(v)^{\frac{N}{N-1}}/v$ is non-increasing (see \cite{EMilman-RoleOfConvexity,EMilmanGeometricApproachPartI,EMilmanNegativeDimension} and the references therein). The proof of this involves crucial use of regularity results from Geometric Measure Theory, and a major challenge is to give a softer proof. In particular, even in the Euclidean setting, an extension of these results to a non-Euclidean boundary measure $\mu^+_L(A)$ is not known and seems technically challenging. The last assertion provides a soft proof for the class of \emph{convex} isoperimetric minimizers, which in fact remains valid for $N < 0$ (cf. \cite{EMilmanNegativeDimension} for \emph{geodesic} extensions). 
\end{rem}

\begin{rem}
As explained in Subsection \ref{subsec:BBL}, it is possible to prove the above assertions using the Borell--Brascamp--Lieb theorem. 
However, this approach would be confined to the Euclidean setting, whereas the proof we give below is not. 
\end{rem}

\begin{proof}[Proof of Proposition \ref{prop:isop-app}]
\hfill
\begin{enumerate}
\item
The first assertion is almost an immediate consequence of the concavity calculation performed in the previous section for the classical Minkowski extension operation $t \mapsto K + t L$. However, in that section we assumed that $K + t L$ is bounded away from the boundary $\partial \Omega$, and we now explain how to remove this restriction. Note that if $y \in \partial K$, then $F_t(y) = y + t \nu^{-1}_L(\nu_K(y))$ is a straight line, as verified in Proposition \ref{prop:Euclidean-Flow}. By the convexity of $\Omega$, this means that this line can at most exit $\Omega$ once, never to return. It is easy to verify that this incurs a non-positive contribution to the calculation of the second variation of $t \mapsto \mu(K + t L)$ in the proof of Theorem \ref{thm:Full-BM}; the rest of the proof remains the same (with the first variation interpreted as the left-derivative). 

More generally, we note here that the concavity statement remains valid if instead of using $\varphi$ which remains constant on the trajectories of the flow, it is allowed to decrease along each trajectory. 
\item
By the concavity from the first assertion, it follows that for every $0<s\leq t$:
\[
N \frac{\mu(K + sL)^{1/N} - \mu(K)^{1/N}}{s} \geq N \frac{\mu(K + tL)^{1/N} - \mu(K)^{1/N}}{t} ~.
\]
Taking the limit as $s \rightarrow 0$, the second assertion follows. 
\item
Given $K \subset \Omega$ with $C^2$ smooth boundary and $\II_{\partial K} > 0$, denote $V(t) := \mu(K + t L )$ and set $\I_{K} := V' \circ V^{-1}$, expressing the boundary measure of $K + t L$ as a function of its measure. Note that $\I_{K}^{\frac{N}{N-1}}(v) / v$ is non-increasing on its domain. Indeed, assuming that $V$ is twice-differentiable, we calculate:
\[
\frac{d}{dv} \frac{\I_{K}^{\frac{N}{N-1}}(v)}{v} = \brac{\brac{\frac{N}{N-1} \frac{V V''}{V'} - V'}\frac{(V')^{\frac{1}{N-1}}}{V^2}} \circ V^{-1}(v)   \leq 0 ~,
\]
and the general case follows by approximation. But since $I^c_L := \inf_{K} \I_K$ where the infimum is over $K$ as above, the third assertion readily follows. 
\end{enumerate}
\end{proof}

\begin{rem}
When in addition $N \mu^{1/N}$ is homogeneous,  i.e. $N \mu(t L)^{1/N} = t N \mu(L)^{1/N}$ for all $t > 0$, it follows by (\ref{eq:for-hom}) that for convex $K$ and $N \in [n,\infty]$:
\[
\mu^+_L(K) \geq \mu(K)^{\frac{N-1}{N}} N \mu(L)^{1/N} ~.
\]
In particular, among all convex sets, homothetic copies of $L$ are isoperimetric minimizers. 
As already eluded to in Subsection \ref{subsec:BBL}, this is actually known to hold for arbitrary Borel sets $K$ (see \cite{CabreRosOtonSerra,EMilmanRotemHomogeneous}). However, this extension from convex to arbitrary Borel sets seems to be a consequence of the homogeneous and linear nature of Euclidean space, and cannot be generalized to a Riemannian setting. 
\end{rem}


\bibliographystyle{plain}

\def\cprime{$'$} \def\textasciitilde{$\sim$}

\end{document}